\documentclass[a4paper, 12pt, reqno]{amsart}
\usepackage[top=32truemm,bottom=28truemm,left=30truemm,right=30truemm]{geometry}

\usepackage{amscd, amsmath, amssymb, amsthm, url, graphicx, bm, blkarray}
\numberwithin{equation}{section}

\theoremstyle{plain}
\newtheorem{thm}{Theorem}[section]
\newtheorem{prop}[thm]{Proposition}
\newtheorem{cor}[thm]{Corollary}
\newtheorem{lem}[thm]{Lemma}

\newtheorem*{thms}{Theorem}

\newtheorem*{thmI}{Theorem I}
\newtheorem*{thmII}{Theorem II}
\newtheorem*{thmIII}{Theorem III}

\theoremstyle{definition}
\newtheorem{defi}[thm]{Definition}
\newtheorem{example}[thm]{Example}
\newtheorem{rem}[thm]{Remark}
\newtheorem*{rem*}{Remark}

\theoremstyle{remark}

\begin{document}

\title[Path group actions induced by sigma-actions]{Path group actions induced by sigma-actions and affine Kac-Moody symmetric spaces of group type}
\author[M. Morimoto]{Masahiro Morimoto}

\address{Osaka Central Advanced Mathematical Institute, Osaka Metropolitan University, 3-3-138 Sugimoto, Sumiyoshi-ku, Osaka, 558-8585, Japan}

\email{mmasahiro0408@gmail.com}

\thanks{The author was partly supported by the Grant-in-Aid for Research Activity Start-up (No.\ 20K22309) and  by Osaka Central Advanced Mathematical Institute (MEXT Joint Usage/Research Center on Mathematics and Theoretical Physics JPMXP06192178499), Osaka Metropolitan University. }

\makeatletter
\@namedef{subjclassname@2020}{%
  \textup{2020} Mathematics Subject Classification}
\makeatother



\maketitle

\begin{abstract}
In 1995, C.-L.\ Terng associated to each hyperpolar action on a compact symmetric space, a hyperpolar proper Fredholm (PF) action on a Hilbert space. This is a group action by an infinite dimensional path group and it acts on a Hilbert space via the gauge transformations. Those two hyperpolar actions are related through an equivariant Riemannian submersion called the parallel transport map and they have close relations to the infinite dimensional symmetric spaces called affine Kac-Moody symmetric spaces. In this paper we define a linear isomorphism between Hilbert spaces and show that it is equivariant with respect to the gauge transformations and is compatible with the parallel transport map. Using this isomorphism we extend and unify all known computational results of principal curvatures of PF submanifolds in Hilbert spaces. Especially we study the submanifold geometry of orbits of hyperpolar PF actions associated to sigma-actions and give new examples of austere PF submanifolds in Hilbert spaces. Moreover we show that the isomorphism between Hilbert spaces given here corresponds to a natural isomorphism between affine Kac-Moody symmetric spaces of group type.
\end{abstract}

\section{Introduction}

An isometric action of a compact Lie group on a Riemannian manifold $M$ is called \emph{polar} if there exists a closed connected submanifold $\Sigma$ of $M$ which meets every orbit and is orthogonal to the orbits at every point of intersection. Such a $\Sigma$ is called a \emph{section}, that is automatically totally geodesic in $M$. If $\Sigma$ is also flat in the induced metric then the action is called \emph{hyperpolar} (\cite{HPTT95}).

If $M$ is a Euclidean space then hyperpolar actions are essentially the isotropy representations of symmetric spaces. More precisely, let $U/L$ be a symmetric space of compact type with canonical decomposition $\mathfrak{u} = \mathfrak{l} + \mathfrak{p}$. The adjoint representation of $L$ on $\mathfrak{p}$ is called the \emph{isotropy representation} of $U/L$. It follows that the isotropy representation is hyperpolar where any maximal abelian subspace in $\mathfrak{p}$ is a section. Conversely it was shown that any hyperpolar representation on a Euclidean space is orbit equivalent to the isotropy representation of a symmetric space (\cite{Dad85, EH99}). Here two isometric actions are called \emph{orbit equivalent} if their orbits are identified via a suitable isometry.

If $M = G/K$ is a symmetric spaces of compact type, important examples of hyperpolar actions are \emph{Hermann actions}, that is, the actions by symmetric subgroups of $G$ (\cite{Her60, Her62}). Here a closed subgroup $H$ of $G$ is called \emph{symmetric} if there exists an involutive automorphism $\theta$ of $G$ which satisfies $G^\theta_0 \subset H \subset G^\theta$, where $G^\theta$ denotes the fixed point subgroup and $G^\theta_0$ its identity component. It was shown that any indecomposable hyperpolar actions of cohomogeneity grater than one is orbit equivalent to a Hermann action (\cite{Kol02, Kol17}). Note that the associated action of $H \times K$ on $G$ defined by $(b,c) \cdot a = bac^{-1}$ is also hyperpolar (\cite{HPTT95}).

A special class of Hermann actions is given by \emph{sigma-actions} (\cite{Con64}). Let $G$ be a connected compact semisimple Lie group and $\sigma$ an automorphism of $G$. Then $G(\sigma) := \{(b, \sigma(b)) \mid b \in G\}$ is a symmetric subgroup of $G \times G$ with involution $(b,c) \mapsto (\sigma^{-1}(c), \sigma(b))$. The $G(\sigma)$-action on $G$ defined by $(b, \sigma(b)) \cdot a = ba \sigma(b)^{-1}$ is called a $\sigma$-action. This can be regarded as a Hermann action by identifying $G$ with the symmetric space $(G \times G) / \Delta G$ where $\Delta G$ denotes the diagonal. The $G(\sigma)$-orbit through the identity is also called the \emph{Cartan embedding} associated to $(G, \sigma)$.

It is an interesting problem to study the submanifold geometry of orbits of hyperpolar actions. The principal orbits of polar representations are isoparametric submanifolds in the sense of Terng \cite{Ter85}. Thorbergsson \cite{Tho91} conversely showed that an irreducible compact full isoparametric submanifold of a Euclidean space with codimension at least $3$ is an orbit of a polar representation. These results were extended to the case of hyperpolar actions on compact symmetric spaces and equifocal submanifolds of symmetric spaces (\cite{TT95, Chr02}). Hirohashi, Song, Takagi and Tasaki \cite{HSTT00} studied the submanifold geometry of orbits of the isotropy representations of symmetric spaces and showed that in each strata of the stratification of orbit types there exists a unique orbit which is a minimal submanifold of the standard sphere. Ikawa \cite{Ika11} extended this result to the case of Hermann actions with commuting involutions. Many geometers have studied orbits of hyperpolar actions and shown various kinds of examples of homogeneous submanifolds.

It is also interesting to study hyperpolar actions in infinite dimensions.

Palais and Terng \cite{PT88, Ter89} introduced a suitable class of isometric actions on Hilbert spaces, namely \emph{proper Fredholm} (PF) actions, and showed examples of hyperpolar PF actions which are orbits of the gauge transformations. These examples were later extended by Pinkall and Thorbergsson \cite{PiTh90} and reformulated by Terng \cite{Ter95} as follows. Let $G$ be a connected compact Lie group with a bi-invariant metric. Denote by $\mathcal{G} = H^1([0,1], G)$ the path group of all Sobolev $H^1$-paths from $[0,1]$ to $G$ and by $V_\mathfrak{g} = H^0([0,1], \mathfrak{g})$ the Hilbert space of all $H^0$-paths from $[0,1]$ to the Lie algebra $\mathfrak{g}$ of $G$. Let $\mathcal{G}$ act on $V_\mathfrak{g}$ by the affine isometry:
\begin{equation*}
g * u = gug^{-1} - g' g^{-1},
\end{equation*}
where $g \in \mathcal{G}$ and $u \in V_\mathfrak{g}$. For any closed subgroup $L$ of $G \times G$ the subgroup
\begin{equation*}
P(G, L) = \{g \in \mathcal{G} \mid (g(0), g(1)) \in L\}
\end{equation*}
acts on $V_\mathfrak{g}$ by the same formula. It was shown that the $P(G, L)$-action is PF and that if the $L$-action on $G$ defined by $(b,c) \cdot a = bac^{-1}$ is hyperpolar then the $P(G,L)$-action on $V_\mathfrak{g}$ is also hyperpolar. Applying this result to the examples of hyperpolar actions on $G$ she showed that $P(G, H \times K)$-actions and $P(G, G(\sigma))$-actions associated to Hermann actions and $\sigma$-actions respectively are hyperpolar. (Palais and Terng \cite{PT88, Ter89} considered the case $\sigma = \operatorname{id}$. Pinkall and Thorbergsson \cite{PiTh90} considered the case $H = K$.) Note that all orbits of $P(G, L)$-actions are \emph{proper Fredholm} (PF) submanifolds of the Hilbert space $V_\mathfrak{g}$ (\cite{PT88, Ter89}). There are many interesting examples of PF submanifolds which are orbits of hyperpolar $P(G, L)$-actions (see, for example \cite{Ter89, Ter95, TT95, M2, M1}).

It should be also noted that there is a class of infinite dimensional symmetric spaces closely related to hyperpolar PF actions. Recall that in the finite dimensional case hyperpolar representations are essentially the isotropy representations of symmetric spaces. Terng \cite{Ter95} conjectured that there is an analogous result in infinite dimensions and showed that $P(G, G(\sigma))$-actions are essentially the adjoint actions of affine Kac-Moody groups. Later Heintze, Palais, Terng and Thorbergsson \cite{HPTT95} studied involutions of affine Kac-Moody algebras and showed that $P(G, H \times K)$-actions associated to Hermann actions are essentially the isotropy representations of infinite dimensional symmetric spaces induced by  those involutions. However they did not give a precise definition of those symmetric spaces due to functional analytic difficulties inherent in affine Kac-Moody groups. Afterward Heintze and Popescu \cite{Hei06, Pop05} started to study those symmetric spaces in the category of tame Fr\'echet manifolds (\cite{Ham82}) and showed their fundamental properties. Nowadays they are called \emph{affine Kac-Moody symmetric spaces} and known as the closest infinite dimensional analogue of finite dimensional symmetric spaces (\cite{Fre11}). According to the theory of those symmetric spaces their isotropy representations (restricted to appropriate subspaces) are essentially equivalent to $P(G, H \times K)$-actions associated to Hermann actions. In particular those of group type are affine Kac-Moody groups and their isotropy representations are essentially equivalent to $P(G, G(\sigma))$-actions.

In the study of $P(G, L)$-actions in general, it is important to consider an equivariant Riemannian submersion $\Phi: V_\mathfrak{g} \rightarrow G$, called the \emph{parallel transport map}. It was shown that each orbit of the $P(G, L)$-action is the inverse image of an $L$-orbit under $\Phi$. More generally, if $N$ is a closed submanifold of $G$ then its inverse image $\Phi^{-1}(N)$ is a PF submanifold of $V_\mathfrak{g}$. For a given compact symmetric space $G/K$ with projection $\pi : G \rightarrow G/K$ we consider the composition $\Phi_K := \pi \circ \Phi : V_\mathfrak{g} \rightarrow G \rightarrow G/K$ which is also an equivariant Riemannian submersion called the parallel transport map over $G/K$. Similarly, if $N$ is a closed submanifold of $G/K$ then its inverse image $\Phi_K^{-1}(N)$ is a PF submanifold of $V_\mathfrak{g}$. In particular, if $N$ is an orbit of the $H$-action then $\Phi_K^{-1}(N)$ is an orbit of the $P(G, H \times K)$-action.  The parallel transport map $\Phi_K$ is also known as a useful tool to study the submanifold geometry in symmetric spaces (\cite{TT95}).

In \cite{M4} the author gave a formula for the principal curvatures of the PF submanifold $\Phi_K^{-1}(N)$ for a curvature-adapted submanifold $N$ of $G/K$ (see also \cite{Koi02, M2}) and showed an explicit formula for the principal curvatures of orbits of $P(G, H \times K)$-actions associated to Hermann actions. Using this formula he studied conditions for those orbits to be austere PF submanifolds of $V_\mathfrak{g}$. Here a submanifold is called \emph{austere} (\cite{HL82}) if for each normal vector $\xi$ the set of  principal curvatures with multiplicities in the direction of $\xi$ is invariant under the multiplication by $(-1)$. Thus austere submanifolds are minimal submanifolds. He considered two conditions: 
\begin{enumerate}
\item[(A)] The orbit $N = H \cdot (\exp w) K$ is an austere submanifold of $G/K$.
\item[(B)] The orbit $\Phi_K^{-1} (N) = P(G, H \times K) * \hat{w}$ is an austere PF submanifold of $V_\mathfrak{g}$.
\end{enumerate}
Here $\hat{w}$ denotes the constant path with value $w \in \mathfrak{g}$. Let $\theta_K$ and $\theta_H$ denote the involutions of $G$ associated to the symmetric subgroups $K$ and $H$ respectively. Denote by $\mathfrak{g} = \mathfrak{k} + \mathfrak{m}$ and $\mathfrak{g} = \mathfrak{h} + \mathfrak{p}$ the corresponding canonical decompositions. Take a maximal abelian subspace $\mathfrak{t}$ in $\mathfrak{m} \cap \mathfrak{p}$ and write $\Delta(\theta_K, \theta_H)$ for the corresponding  root system of $\mathfrak{t}$. He showed:
\begin{thms}[{\cite{M4}}] \  
\begin{enumerate}
\item If $\Delta(\theta_K, \theta_H)$ is a reduced root system then \textup{(A)} and \textup{(B)} are equivalent.
\item If $\theta_K = \theta_H$ then \textup{(A)} and \textup{(B)} are equivalent.
\item If $\theta_K \circ \theta_H = \theta_H \circ \theta_K$ then \textup{(A)} implies \textup{(B)}.
\item If $G$ is simple then \textup{(A)} implies \textup{(B)}.
\end{enumerate}
Here \textup{(B)} does not imply \textup{(A)} in the cases \textup{(iii)} and \textup{(iv)}. In fact, there exists a counterexample.
\end{thms}
\noindent
Applying these results to the examples of austere orbits of Hermann actions he showed many examples of austere PF submanifolds which are orbits of hyperpolar $P(G, H \times K)$-actions.

The main purpose of this paper is to extend those results to the case of $P(G, G(\sigma))$-actions. Notice that although the $\sigma$-action is a special case of a Hermann action, we can \emph{not} apply the previous results directly to the present case because $G(\sigma)$ is not the product of two symmetric subgroups of $G$. In this paper we introduce an injective homomorphism $\Omega: H^1([0,1], G) \rightarrow H^1([0,1], G \times G)$ and a linear isomorphism $\Upsilon: H^0([0,1], \mathfrak{g}) \rightarrow H^0([0,1], \mathfrak{g} \oplus \mathfrak{g})$, and show (Theorem \ref{thm1} and Corollary \ref{cor3.1}): 
\begin{thmI} \ \
\begin{enumerate}
\item The $P(G, L)$-action on $V_\mathfrak{g}$ is conjugate to the $P(G \times G, L \times \Delta G)$-action on $V_{\mathfrak{g} \oplus \mathfrak{g}}$ via $\Omega$ and $\Upsilon$, that is, $\Omega$ maps $P(G, L)$ isomorphically onto $P(G \times G, L \times \Delta G)$ and $\Upsilon(g * u ) = \Omega(g) * \Upsilon(u)$ holds for $g \in P(G, L)$ and $u \in V_\mathfrak{g}$. 
In particular the $P(G, G(\sigma))$-action on $V_\mathfrak{g}$ is conjugate to the $P(G \times G, G(\sigma) \times \Delta G)$-action on $V_{\mathfrak{g} \oplus \mathfrak{g}}$ via $\Omega$ and $\Upsilon$.

\item The following diagram commutes:
\begin{equation*}
\begin{CD}
V_\mathfrak{g} @>\Upsilon>> V_{\mathfrak{g} \oplus \mathfrak{g}}
\\
@V\Phi VV @V \Phi_{\Delta G} VV
\\
G @<\phi<< (G \times G) /\Delta G\,,
\end{CD}
\end{equation*}
where $\Phi_{\Delta G}: V_{\mathfrak{g}\oplus \mathfrak{g}} \rightarrow G \times G \rightarrow (G \times G) /\Delta G$ denotes the parallel transport map over $(G \times G) /\Delta G$ and $\phi$ the isomorphism $(a,b) \mapsto ab^{-1}$.
\end{enumerate}
\end{thmI}
\noindent
The property (i) allows us to apply the general results of $P(G, H \times K)$-actions to $P(G, G(\sigma))$-actions. Note that the case $\sigma = \operatorname{id}$ was essentially observed by Pinkall and Thorbergsson \cite[p.\ 283]{PiTh90}. The property (ii) means that $\Upsilon$ is natural in the framework of parallel transport maps. It allows us to apply the general results of $\Phi_K$ to $\Phi$.

Using Theorem I we derive a formula for the principal curvatures of the PF submanifold $\Phi^{-1}(N)$ for a curvature-adapted submanifold $N$ of $G$ (Theorem \ref{thm2}) and a formula for the principal curvatures of orbits of $P(G, G(\sigma))$-actions (Theorem \ref{thm3}).  These formulas generalize the results by King-Terng \cite{KT93} in the case of fibers and by Palais-Terng \cite{PT88, Ter89} in the case of $\sigma = \operatorname{id}$. Consequently we unify all known computational results of principal curvatures of PF submanifolds as follows:
\begin{equation*}
\begin{array}{ccc}
	\begin{array}{c}
	\text{Fiber $\Phi_K^{-1}(aK)$}
	\\
	\text{(The author \cite{M2})}
	\end{array}
		&\quad \overset{\text{Theorem I}}{\Longrightarrow} \quad&
	\begin{array}{c}
	\text{Fiber $\Phi^{-1}(a)$}
	\\
	\text{(King-Terng \cite{KT93})}
	\end{array}
	\medskip
\\
	\rotatebox{90}{\hspace{-2mm}$\Longrightarrow$}&&\rotatebox{90}{\hspace{-2mm}$\Longrightarrow$}
	\vspace{-5mm}
\\
	\hspace{25mm} \text{\footnotesize $N = \{aK\}$}&&\hspace{23mm} \text{\footnotesize $N = \{a\}$}
	\bigskip
\\
	\begin{array}{c}
	\text{PF submanifold $\Phi_K^{-1}(N)$}
	\\
	\text{(Koike \cite{Koi02}, the author \cite{M2, M4})}
	\end{array}
		&\quad \overset{\text{Theorem I}}{\Longrightarrow} \quad&
	\begin{array}{c}
	\text{PF submanifold $\Phi^{-1}(N)$}
	\\
	\text{(Theorem \ref{thm2} of this paper)}
	\end{array}
	\vspace{-1mm}
\\
		\rotatebox{-90}{$\Longrightarrow$}&&\rotatebox{-90}{$\Longrightarrow$}
	\vspace{-6mm}
\\
	\hspace{27mm} \text{\footnotesize $N = H \cdot (\exp w)K $}\hspace{-10mm}&&\hspace{27mm} \text{\footnotesize $N = G(\sigma) \cdot \exp w$\hspace{-10mm}}
	\bigskip
\\
	\begin{array}{c}
	\text{Orbit $P(G, H \times K) * \hat{w}$}
	\\
	\text{(The author \cite{M4})}
	\end{array}
		&\quad\overset{\text{Theorem I}}{\Longrightarrow}\quad&
	\begin{array}{c}
	\text{Orbit $P(G, G(\sigma)) * \hat{w}$}
	\\
	\text{(Theorem \ref{thm3} of this paper)}
	\end{array}
	\vspace{-1mm}
\\	
		\rotatebox{-90}{$\Longrightarrow$}&&\rotatebox{-90}{$\Longrightarrow$}
	\vspace{-6mm}
\\
	\hspace{20mm} \text{\footnotesize $H = K$}&&\hspace{20mm} \text{\footnotesize $\sigma = \operatorname{id}$}
	\bigskip
\\
	\begin{array}{c}
	\text{Orbit $P(G, K \times K) * \hat{w}$}
	\\
	\text{(Pinkall-Thorbergsson \cite{PiTh90})}
	\end{array}
		&\quad\overset{\text{Theorem I}}{\Longrightarrow}\quad&
\begin{array}{c}
	\text{Orbit $P(G, \Delta G) * \hat{w}$}
	\\
	\text{(Palais-Terng \cite{PT88, Ter89}).}
	\end{array}
\end{array}
\end{equation*}

Based on those results we study the relation between the following two conditions on the austere property of orbits:
\begin{enumerate}
\item[(a)] The orbit $N = G(\sigma) \cdot \exp w$ is an austere submanifold of $G$.
\item[(b)] The orbit $\Phi^{-1}(N) = P(G, G(\sigma)) * \hat{w}$ is an austere PF submanifold of $V_\mathfrak{g}$.
\end{enumerate}
Let $\mathfrak{t}$ be a maximal abelian subalgebra of the fixed point algebra $\mathfrak{g}^\sigma$ and $\Delta(\sigma)$ the corresponding root system of $\mathfrak{t}$ (see Section \ref{pcosa}). We prove (Theorem \ref{mainthm}):
\begin{thmII} \ 
\begin{enumerate}
\item If $\Delta(\sigma)$ is a reduced root system then \textup{(a)} and \textup{(b)} are equivalent. 
\item If $\sigma = \operatorname{id}$ then \textup{(a)} and \textup{(b)} are equivalent.
\item If $\sigma^2 = \operatorname{id}$ then \textup{(a)} implies \textup{(b)}.
\item If $G$ is simple then \textup{(a)} implies \textup{(b)}. 
\end{enumerate}
Here \textup{(b)} does not imply \textup{(a)} in the cases \textup{(iii)} and \textup{(iv)}. In fact, there exists a counterexample.
\end{thmII}
\noindent
This is an analogue of the previous theorem (\cite{M4}). In fact it turns out by Theorem I that (a) and (b) are special cases of (A) and (B) respectively and thus (i)--(iii) follow from the previous results. However (iv) is \emph{not} trivial because the simplicity is not preserved by $\Omega$ and $\Upsilon$. Moreover the converse is \emph{not} trivial because the counterexample given in the previous paper is not an orbit of a $\sigma$-action. We prove (iv) based on the structure theory of automorphisms of $G$ and show a counterexample to the converse (Theorem \ref{prop9.2}). We also extend the author's previous results concerning weakly reflective PF submanifolds (Theorem \ref{thm5}).

Finally we study the relations to affine Kac-Moody symmetric spaces. Recall that each finite dimensional Lie group $G$ is regarded as the symmetric space $(G \times G) / \Delta G$. Similarly each affine Kac-Moody group $\hat{G}$ can be regraded as a symmetric space $\widehat{G \times G} / (\widehat{G \times G})^{\hat{\tau}}$ via a certain isomorphism $\Lambda$. We know that the isotropy representation of $\hat{G}$ is essentially the $P(G, G(\sigma))$-action on $V_\mathfrak{g}$ (\cite{Ter95}). We will prove that the isotropy representation of $\widehat{G \times G} / (\widehat{G \times G})^{\hat{\tau}}$ is essentially the $P(G \times G, G(\sigma) \times \Delta G)$-action on $V_{\mathfrak{g} \oplus \mathfrak{g}}$ (Proposition \ref{cor10.3}). Those two actions are conjugate via the isomorphisms $(\Omega, \Upsilon)$ and related to $\sigma$-actions via the parallel transport map (Theorem I). Consequently we show (Theorem \ref{final}):
\begin{thmIII}
There is a correspondence between the isomorphisms $\Lambda$, $(\Omega, \Upsilon)$ and $(\operatorname{id}, \phi)$:%
\smallskip
\begin{equation*}
\hspace{-2mm}
\begin{array}{rccc}
\begin{array}{r}
\textup{\small Affine Kac-Moody}\vspace{-0.5mm} \\ \textup{\small symmetric space}\vspace{-1mm}
\end{array}
\quad&
\hat{G} = \hat{L}(G, \sigma) 
& \overset{\Lambda}{\cong}& \widehat{G \times G} / (\widehat{G \times G})^{\hat{\tau}}
\bigskip
\\
\begin{array}{r}
\textup{\small Isotropy}\vspace{-0.5mm} \\ \textup{\small representation}
\end{array}
\quad&
P(G,G(\sigma)) \curvearrowright V_\mathfrak{g} & \quad \overset{(\Omega,\, \Upsilon)}{\cong} \quad & P(G \times G, G(\sigma) \times \Delta G) \curvearrowright V_{\mathfrak{g} \oplus \mathfrak{g}}
\bigskip
\\
\begin{array}{r}
\textup{\small Finite dimensional}\vspace{-0.5mm}\\ \textup{\small counterpart}
\end{array}
\quad&
G(\sigma) \curvearrowright G& \overset{(\operatorname{id},\, \phi)}{\cong} & G(\sigma) \curvearrowright (G \times G) / \Delta G.
\end{array}
\smallskip
\end{equation*}
\end{thmIII}

This paper is organized as follows.
In Section \ref{prel} we review foundations of PF submanifolds, PF actions and parallel transport maps.
In Section \ref{paral} we define and investigate the isomorphisms $\Omega$ and $\Upsilon$ and prove Theorem I.
In Section \ref{pcvp} we derive a formula for the principal curvatures of the PF submanifold $\Phi^{-1}(N)$ for a curvature-adapted submanifold $N$ of $G$.
In Section \ref{pcosa} we study the submanifold geometry of orbits of $\sigma$-actions. 
In Section \ref{pco} we derive an explicit formula for the principal curvatures of $P(G, G(\sigma))$-orbits.
In Section \ref{ausp} we study conditions for $P(G, G(\sigma))$-orbits to be austere PF submanifolds of $V_\mathfrak{g}$ and prove Theorem II. 
In Section \ref{weakp} we extend the previous results concerning weakly reflective PF submanifolds in Hilbert spaces.
In Section \ref{eg} we show concrete examples of austere PF submanifolds and weakly reflective PF submanifolds which are orbits of a $P(G, G(\sigma))$-action.
In Section \ref{AKM} we study the relations to affine Kac-Moody symmetric spaces and prove Theorem III.

\section{PF submanifolds, PF actions and parallel transport maps} \label{prel}

In this section we review foundations of PF submanifolds, PF actions and parallel transport maps.  

Let $N$ be a submanifold of a (separable) Hilbert space $V$. Suppose that $N$ has finite codimension in $V$. $N$ is called \emph{proper Fredholm} (PF) if the end point map $T^\perp N \rightarrow V$, $(p, \xi) \mapsto p + \xi$ restricted to a normal disc bundle of any finite radius is proper and Fredholm (\cite{Ter89}). The proper condition implies that for each $u \in V$ the function $f_u: N \rightarrow \mathbb{R}$, $p \mapsto \|p - u\|^2$ satisfies the Palais-Smale condition (\cite{Pal63, Sma64}). The Fredholm condition implies that the shape operators are compact self-adjoint operators. 

Let $\mathcal{L}$ be a Hilbert Lie group, acting on a Hilbert space $V$. The action is called \emph{proper Fredholm} (PF) if the map $\mathcal{L} \times V \rightarrow V \times V$, $(l, u) \mapsto (l \cdot u, u)$ is proper and the map $\mathcal{L} \rightarrow V$, $l \mapsto l \cdot u$ is Fredholm for each $u \in V$ (\cite{PT88}). If $\mathcal{L}$ is infinite dimensional and the action is isometric PF, then every $\mathcal{L}$-orbit is a PF submanifold of $V$ (\cite[Theorem 7.1.6]{PT88}). 

Let $G$ be a connected compact Lie group with Lie algebra $\mathfrak{g}$. Choose an $\operatorname{Ad}(G)$-invariant inner product $\langle \cdot, \cdot \rangle$ of $\mathfrak{g}$ and equip the corresponding bi-invariant Riemannian metric with $G$.  Denote by $\mathcal{G} = H^1([0,1], G)$ the path group of all Sobolev $H^1$-paths from $[0,1]$ to $G$ and by $V_\mathfrak{g} = H^0([0,1],\mathfrak{g})$ the Hilbert space of all $H^0$-paths from $[0,1]$ to $\mathfrak{g}$. Then $\mathcal{G}$ acts on $V_\mathfrak{g}$ by the affine isometry:
\begin{equation*}
g * u = gug^{-1} - g' g^{-1},
\end{equation*}
where $g \in \mathcal{G}$, $u \in V_\mathfrak{g}$ and $g'$ denotes the weak derivative of $g$. It follows that this action is transitive and PF (\cite{Ter95}).

Let $L$ be a closed subgroup of $G \times G$. The subgroup
\begin{equation*}
P(G, L) = \{g \in \mathcal{G} \mid (g(0) , g(1)) \in L\}
\end{equation*}
acts on $V_\mathfrak{g}$ by the same formula. Note that $P(G, L)$ is the inverse image of $L$ under the submersion $\Psi^G: \mathcal{G} \rightarrow G \times G$ defined by 
\begin{equation*}
\Psi^G(g) = (g(0), g(1)).
\end{equation*}
It follows that the $P(G, L)$-action is also PF (\cite{Ter95}). Thus every orbit of the $P(G, L)$-action is a PF submanifold of $V_\mathfrak{g}$. 

For each $u \in V_\mathfrak{g}$ we define $g_u \in \mathcal{G}$ as the unique solution to the linear ordinary differential equation
\begin{equation*}
g^{-1} g' = u,
\quad
g(0) = e.
\end{equation*}
The \emph{parallel transport map} $\Phi: V_\mathfrak{g} \rightarrow G$ is a Riemannian submersion defined by 
\begin{equation*}
\Phi(u) = g_u(1).
\end{equation*}
By definition $\Phi(\hat{x}) = \exp x$ where $\hat{x}$ denotes the constant path with value $x \in \mathfrak{g}$. Consider the action of $G \times G$ on $G$ by
\begin{equation}\label{action1}
(b,c) \cdot a = bac^{-1}.
\end{equation}
Then $\Phi$ is equivariant via $\Psi^G$, that is, 
\begin{equation*}
\Phi( g * u) = (g(0), g(1)) \cdot \Phi(u)
\end{equation*}
for $g \in \mathcal{G}$ and $u \in V_\mathfrak{g}$. Moreover it follows that
\begin{equation*}
P(G, L) * u = \Phi^{-1}(L \cdot \Phi(u))
\end{equation*}
for any closed subgroup $L$ of $G \times G$ (\cite{Ter95}). More generally, if $N$ is a closed submanifold of $G$ then the inverse image $\Phi^{-1}(N)$ is a PF submanifold of $V_\mathfrak{g}$ (\cite[Lemma 5.8]{TT95}).

Let $K$ be a symmetric subgroup of $G$ with Lie algebra $\mathfrak{k}$. Denote by $\mathfrak{g} = \mathfrak{k} + \mathfrak{m}$ the decomposition into the $(\pm1)$-eigenspaces of the involution, which is called the canonical decomposition. Restricting the $\operatorname{Ad}(G)$-invariant inner product of $\mathfrak{g}$ to $\mathfrak{m}$ we equip the induced $G$-invariant Riemannian metric with the homogeneous space $G/K$. Then $G/K$ is a compact symmetric space and the natural projection $\pi :  G \rightarrow G/K$ is a Riemannian submersion with totally geodesic fiber. The composition 
\begin{equation*}
\Phi_K : = \pi \circ \Phi : V_\mathfrak{g} \rightarrow G \rightarrow G/K
\end{equation*}
is a Riemannian submersion which is called the \emph{parallel transport map over $G/K$}. Consider the action of $G$ on $G/K$ by
\begin{equation}\label{action2}
b \cdot (aK) = (b a) K.
\end{equation}
Denote by $p^G: G \times G \rightarrow G$ the projection onto the first component. Then  $\Phi_K$ is equivariant via $p^G \circ \Psi^G$, that is, 
\begin{equation*}
\Phi_K(g * u) = g(0) \cdot \Phi_K(u)
\end{equation*}
for $g \in P(G, G \times K)$. Moreover we have
\begin{equation*}
P(G, H \times K) * u = \Phi_K^{-1}(H \cdot \Phi_K(u))
\end{equation*}
for any closed subgroup $H$ of $G$. More generally, if $N$ is a closed submanifold of $G/K$ then the inverse image $\Phi_K^{-1}(N)$ is a PF submanifold of $V_\mathfrak{g}$.

\section{The canonical isomorphism of path spaces}\label{paral}

In this section we introduce an isomorphism between path spaces and investigate their relations to the $P(G, L)$-actions and parallel transport maps.

Recall that a connected compact Lie group $G$ with a bi-invariant metric is regarded as the symmetric space $(G \times G) /\Delta G$. In fact the diagonal $\Delta G$ is a symmetric subgroup of $G \times G$ with involution $(b,c) \mapsto (c,b)$. The canonical decomposition is given by 
\begin{equation*}
\mathfrak{g} \oplus \mathfrak{g} = \mathfrak{k} + \mathfrak{m},
\end{equation*}
where $\mathfrak{k} = \Delta \mathfrak{g} = \{(x,x) \mid x \in \mathfrak{g}\}$ and $\mathfrak{m} = (\Delta \mathfrak{g})^\perp = \{(x, - x) \mid x \in \mathfrak{g}\}$.  Consider the diffeomorphism
\begin{equation*}
\phi: (G \times G) / \Delta G \rightarrow G
, \quad
(b, c) \Delta G \mapsto bc^{-1},
\end{equation*}
whose differential at $(e,e) \Delta G$ is identified with the map
\begin{equation*}
d\phi : (\Delta \mathfrak{g})^\perp \rightarrow \mathfrak{g}
, \quad
(x, - x) \mapsto  2x.
\end{equation*}
Note that 
\begin{equation*}
\langle d \phi (x, -x) , d \phi (y, - y)\rangle  = 2 \langle (x, -x), (y, -y) \rangle.
\end{equation*}
Note also that $G \times G$ acts on $G$ by \eqref{action1} and acts also on $(G \times G)/ \Delta G$ by \eqref{action2}. Clearly $\phi$ is equivariant with respect to these $G \times G$-actions.

There is a natural isomorphism of path spaces corresponding to $\phi$:

\begin{defi}\label{cannicalisom}
Define the injective homomorphism $\Omega: \mathcal{G} \rightarrow H^1([0,1], G \times G)$ by 
\begin{equation*}
\Omega (g) = (g(t/2), g(1 - t/2)),
\end{equation*}
and the linear isomorphism $\Upsilon: V_\mathfrak{g} \rightarrow V_{\mathfrak{g}\oplus \mathfrak{g}}$ by
\begin{equation*}
\Upsilon (u) = (\ \frac{1}{2} u(t/2), \  -\frac{1}{2}u(1 - t/2) \ ).
\end{equation*}
We call $\Upsilon$ the \emph{canonical isomorphism} from $V_\mathfrak{g}$ to $V_{\mathfrak{g} \oplus \mathfrak{g}}$. We also call $\Omega$ the canonical isomorphism (from $\mathcal{G}$ to $\Omega(\mathcal{G})$) if there is no confusion.
\end{defi}

It is easy to see that 
\begin{equation*}
\langle \Upsilon(u), \Upsilon(v) \rangle_{L^2} = \frac{1}{2} \langle u, v\rangle_{L^2}.
\end{equation*}
The maps $\Omega$ and $\Upsilon$ have the following equivariant properties: 

\begin{thm}\label{thm1} \ 
\begin{enumerate}
\item $\Omega(P(G, L)) = P(G \times G, L \times \Delta G)$ for a closed subgroup $L$ of $G \times G$. In particular the image of $\Omega$ is $P(G \times G, G \times G \times \Delta G)$.

\item $\Upsilon$ is equivariant via $\Omega$, that is, 
\begin{equation*}
\Upsilon(g * u ) = \Omega(g) * \Upsilon(u)
\end{equation*}
for $g \in \mathcal{G}$ and $u \in V_\mathfrak{g}$.

\item The following diagrams are commutative:
\begin{equation*}
\begin{CD}
\mathcal{G} @>\Omega >> H^1([0,1], G \times G)
\\
@V \Psi^G VV @V p^{G \times G} \circ \Psi^{G \times G}VV 
\\
G \times G @>\operatorname{id} >> G \times G
\end{CD}
\quad \text{and} \qquad
\begin{CD}
V_\mathfrak{g} @>\Upsilon>> V_{\mathfrak{g} \oplus \mathfrak{g}}
\\
@V\Phi VV @V \Phi_{\Delta G} VV
\\
G @<\phi<< (G \times G) /\Delta G\,,
\end{CD}
\end{equation*}
where $\Phi_{\Delta G}: V_{\mathfrak{g}\oplus \mathfrak{g}} \rightarrow G \times G \rightarrow (G \times G) /\Delta G$ denotes the parallel transport map over the symmetric space $(G \times G) /\Delta G$.
\end{enumerate}
\end{thm}

\begin{proof}
(i): Clearly $\Omega(g)(0) = (g(0), g(1))$ and $\Omega(g)(1) = (g(1/2) , g(1/2)) \in \Delta G$. Thus $g \in P(G, L)$ if and only if $\Omega(g) \in P(G \times G, L \times \Delta G)$. Conversely every element of $P(G \times G, L \times \Delta G)$ is obtained in this way. This proves (i).

(ii): We set $\Omega(g) = \tilde{g} = (\tilde{g}_1, \tilde{g}_2)$ and $\Upsilon(u) = \tilde{u} = (\tilde{u}_1, \tilde{u}_2)$ so that $\tilde{g}_1(t) = g(t/2)$, $\tilde{g}_2(t) = g(1 - t/2)$, $\tilde{u}_1 =  \frac{1}{2} u(t/2)$, $\tilde{u}_2 = - \frac{1}{2} u(1 - t/2)$. Then we have
\begin{align*}
\Upsilon(gug^{-1}) (t)
&=
(\ \frac{1}{2} g(t/2)u(t/2)g(t/2)^{-1}, - \frac{1}{2} g(1 - t/2)u(1 - t/2)g(1 - t/2)^{-1}\ )
\\
&=
(\ \tilde{g}_1(t) \tilde{u}_1(t) \tilde{g}_1(t)^{-1}, \tilde{g}_2(t)\tilde{u}_2(t) \tilde{g}_2(t)^{-1}\ )
\\
&=
(\tilde{g}_1, \tilde{g}_2) (\tilde{u}_1, \tilde{u}_2) (\tilde{g}_1, \tilde{g}_2)^{-1}(t)
=
\Omega(g) \Upsilon(u) \Omega(g)^{-1} (t).
\end{align*}
Moreover since $\tilde{g}_1'(t) =\frac{1}{2} g'(t/2)$ and $\tilde{g}_2'(t) = - \frac{1}{2} g'(1 - t/2)$ we have
\begin{align*}
\Upsilon(g'g^{-1})(t)
&= 
(\frac{1}{2} g'(t/2)g(t/2)^{-1}, - \frac{1}{2} g'(1 - t/2)g(1 -t/2)^{-1})
\\
&=
(\tilde{g}_1'(t) \tilde{g}_1(t)^{-1},  \tilde{g}_2'(t) \tilde{g}_2(t)^{-1})
=
(\tilde{g}_1, \tilde{g}_2)' (\tilde{g}_1, \tilde{g}_2)^{-1}
=
\Omega(g)' \Omega(g)^{-1}(t).
\end{align*}
Therefore we have 
\begin{align*}
\Upsilon(g * u) 
&= 
\Upsilon(g ug^{-1} - g'g^{-1})
=
\Upsilon(g ug^{-1}) -  \Upsilon(g'g^{-1})
\\
&=
\Omega(g) \Upsilon(u)\Omega(g)^{-1} -  \Omega(g)'\Omega(g)^{-1}
=
\Omega(g) * \Upsilon(u).
\end{align*}
This proves (ii).

(iii): Let $g \in \mathcal{G}$. Then we have
\begin{align*}
p^{G \times G} \circ \Psi^{G \times G} \circ \Omega(g)
&=
p^{G \times G} \circ \Psi^{G \times G} (\tilde{g})
=
p^{G \times G} (\tilde{g}(0), \tilde{g}(1))
\\
&=
\tilde{g}(0) = (\tilde{g}_1(0), \tilde{g}_2(0)) = (g(0), g(1)) = \Psi^G(g).
\end{align*}
Let $u \in V_\mathfrak{g}$. Take $h \in \mathcal{G}$ satisfying $u = h * \hat{0}$. Then we have
\begin{align*}
\Phi_{\Delta G}(\Upsilon(u)) 
& = 
\Phi_{\Delta G}(\Omega(h) * (\hat{0}, \hat{0}) )
= 
(p^{G \times G} \circ \Psi^{G \times G})(h) \cdot \Phi_{\Delta G} (\hat{0}, \hat{0})
\\
&=
(h(0), h(1)) \Delta G 
=
\phi^{-1}(h(0)h(1)^{-1}) = \phi^{-1}(\Phi(u)).
\end{align*}
This proves (iii)
\end{proof}

Two isometric actions $A_1$ on $X_1$ and $A_2$ on $X_2$ are said to be \emph{conjugate} if there exist an isomorphism $\phi: A_1 \rightarrow A_2$ and an isometry $\varphi:X_1\rightarrow X_2$ satisfying $ \varphi(a \cdot p) = \phi(a) \cdot \varphi(p)$ for $a \in A_1$ and $p \in X_1$. In this case we say that these actions are conjugate \emph{via $\phi$ and $\varphi$}. We allow $\varphi$ to be a diffeomorphism such that $\lambda \varphi$ is an isometry for a suitable $\lambda \in \mathbb{R}$. 

\begin{cor}\label{cor3.1}
The $P(G, L)$-action on $V_\mathfrak{g}$ is conjugate to the $P(G \times G, L \times \Delta G)$-action on $V_{\mathfrak{g} \oplus \mathfrak{g}}$ via $\Omega$ and $\Upsilon$. In particular the $P(G, G(\sigma))$-action on $V_\mathfrak{g}$ is conjugate to the $P(G \times G, G(\sigma) \times \Delta G)$-action on $V_{\mathfrak{g} \oplus \mathfrak{g}}$ via $\Omega$ and $\Upsilon$.
\end{cor}
\noindent

Corollary \ref{cor3.1} is a consequence of (i) and (ii) of Theorem \ref{thm1}. It allows us to apply the general results of $P(G, H \times K)$-actions to $P(G, G(\sigma))$-actions. (iii) of Theorem \ref{thm1} allows us to apply the general results of $\Phi_K$ to $\Phi$.

\begin{rem}\label{PhThrmk}
Considering the case $\sigma = \operatorname{id}$ in Corollary \ref{cor3.1} we see that the $P(G, \Delta G)$-action on $V_\mathfrak{g}$ is conjugate to the $P(G \times G, \Delta G \times \Delta G)$-action on $V_{\mathfrak{g} \oplus \mathfrak{g}}$. This fact was essentially observed by Pinkall and Thorbergsson \cite[Remark in p.\ 283]{PiTh90}.
\end{rem}

\section{Principal curvatures via the parallel transport map $\Phi$} \label{pcvp}

In this section we derive a formula for the principal curvatures of the PF submanifold $\Phi^{-1}(N)$ for a curvature-adapted submanifold $N$ of $G$. 

Let $G$ be a connected compact semisimple Lie group with a bi-invariant metric, $\Phi: V_\mathfrak{g} \rightarrow G$ the parallel transport map and $N$ a  closed submanifold of $G$. Suppose that $N$ is \emph{$k$-curvature-adapted} (\cite{M4}), that is,  for each $a \in N$ the following conditions hold:
\begin{enumerate}
\item for every $v \in T^\perp_a N$ the curvature operator $R_v$ leaves $T_a N$ invariant,
\item for each $v \in T^\perp_a N$ there exists a $k$-dimensional abelian subalgebra $\mathfrak{t}$ of $\mathfrak{g}$ satisfying $v \in dl_a(\mathfrak{t}) \subset T^\perp_a N$ such that
\begin{equation*}
\{R_{dl_a(\xi)} |_{T_a N}\}_{\xi \in \mathfrak{t}} \cup \{A^N_{dl_a(\xi)}\}_{\xi \in \mathfrak{t}}
\end{equation*}
is a commuting family of endomorphisms of $T_{a}N$.
\end{enumerate}
Here the curvature operator $R_v$ is defined by $R_v(x) = R^G(x, v)v$ where $R^G$ denotes the curvature tensor of $G$. If $a = e$ then $R_v$ is identified with $- \frac{1}{4} \operatorname{ad}(v)^2$. Clearly $1$-curvature-adapted submanifolds are just curvature-adapted submanifolds in the original sense (\cite{BV92}).  We know that all orbits of sigma-actions of cohomogeneity $k$ are $k$-curvature-adapted submanifolds (\cite[Corollaries 3.3 and 3.4]{GT07}, \cite[Proposition 4.2]{M4}).

By left translations we can assume without loss of generality that $N$ is through $e \in G$. Choose and fix a $k$-dimensional abelian subalgebra $\mathfrak{t}$ of $\mathfrak{g}$ satisfying the above condition (ii) and consider the real root space decomposition with respect to $\mathfrak{t}$:
\begin{equation*}
\mathfrak{g} = \mathfrak{g}_0 + \sum_{\alpha \in \Delta^+} \mathfrak{g}_\alpha,
\end{equation*}
where
\begin{align*}
\mathfrak{g}_0
&=
\{x \in \mathfrak{g} \mid \operatorname{ad}(\eta) x = 0 \ \ \text{for all} \ \eta \in \mathfrak{t}  \},
\\
\mathfrak{g}_\alpha
&=
\{x \in \mathfrak{g} \mid  \operatorname{ad}(\eta)^2 x = - \langle \alpha, \eta \rangle ^2 x \ \ \text{for all} \ \eta \in \mathfrak{t} \}.
\end{align*}
This is the common eigenspace decomposition of the commuting operators $\{\operatorname{ad}(\xi)^2\}_{\xi \in \mathfrak{t}}$. On the other hand, we have the common eigenspace decomposition of the commuting shape operators $\{A^N_\xi\}_{\xi \in \mathfrak{t}}$. More precisely there exits a unique finite subset $\Lambda$ of $\mathfrak{t}$ such that (\cite[Lemma 4.4]{M4})
\begin{equation*}
T_e N = \sum_{\lambda \in \Lambda} S_\lambda,
\end{equation*}
where
\begin{equation*}
S_\lambda = \{x \in T_e N \mid A^N_{\xi} (x) = \langle \lambda, \xi \rangle x \ \text{for all} \ \xi \in \mathfrak{t}\}.
\end{equation*}
Since all those operators commute we have 
\begin{align*}
T_e N &= \sum_{\lambda \in \Lambda_0} (\mathfrak{g}_0 \cap S_\lambda) + \sum_{\alpha \in \Delta^+} \sum_{\lambda \in \Lambda_\alpha} (\mathfrak{g}_\alpha \cap S_\lambda),
\\
T^\perp_e N &= \mathfrak{g}_0 \cap T^\perp_e N + \sum_{\alpha \in \Delta^+} (\mathfrak{g}_\alpha \cap T^\perp_e N),
\end{align*}
where $\Lambda_0 = \{\lambda \in \Lambda \mid \mathfrak{g}_0 \cap S_\lambda\}$ and $\Lambda_\alpha = \{\lambda \in \Lambda \mid \mathfrak{g}_\alpha \cap S_\lambda \neq \{0\}\}$. Set 
\begin{equation*}
\begin{array}{lcl}
m(0, \lambda) = \dim (\mathfrak{g}_\alpha \cap S_\lambda), 
&&
m(\alpha, \lambda) = \dim (\mathfrak{g}_\alpha \cap S_\lambda), 
\medskip
\\
m(0, \perp ) = \dim (\mathfrak{g}_0 \cap T^\perp_e N),
&&
m(\alpha, \perp) = \dim (\mathfrak{g}_\alpha \cap T^\perp_e N).
\end{array}
\end{equation*}

Based on these decompositions we can describe the principal curvatures of the PF submanifold $\Phi^{-1}(N)$:
\begin{thm}\label{thm2}
Let $G$ be a connected compact semisimple Lie group with a bi-invariant metric, $N$ a $k$-curvature-adapted submanifold of $G$ through $e \in G$ and $\mathfrak{t}$ an abelian subalgebra of $\mathfrak{g}$ satisfying the above condition \textup{(ii)}. Then for each $\xi \in \mathfrak{t}$ the principal curvatures of $\Phi^{-1}(N)$ in the direction of $\hat{\xi}$ are given by
\begin{align*}
\{0\}
& \cup
\{\langle \lambda ,\xi \rangle 
\mid
\lambda \in \Lambda_0 \cup {\textstyle \bigcup_{\beta \in \Delta^+_\xi }} \Lambda_\beta \}
\\
& \cup
\left\{
\left.
\frac{\langle \alpha, \xi \rangle}{2 \arctan \frac{\langle \alpha, \xi \rangle}{2 \langle \lambda, \xi \rangle}+ 2 m \pi}
\ \right|\ 
\alpha \in \Delta^+ \backslash \Delta^+_\xi , \ \lambda \in \Lambda_\alpha, \ m \in \mathbb{Z} \right\}
\\
& \cup
\left\{
\left.
\frac{\langle \alpha, \xi \rangle}{2 n \pi}
\ \right|\ 
\alpha \in \Delta^+ \backslash \Delta^+_\xi , \ \mathfrak{g}_\alpha \cap T^\perp_{eK} N \neq \{0\}, \  n \in \mathbb{Z} \backslash \{0\}
\right\},
\end{align*}
where we set $\Delta^+_\xi := \{\beta \in \Delta^+ \mid \langle \beta , \xi \rangle = 0\}$ and $\arctan \frac{\langle \alpha, \xi \rangle}{2 \langle \lambda, \xi \rangle} := \frac{\pi}{2}$ if $\langle \lambda, \xi \rangle = 0$. The multiplicities are respectively given by
\begin{equation*}
\infty
, \qquad
m(0, \lambda) + \sum_{\beta \in \Delta_\xi} m(\beta, \lambda)
, \qquad
m(\alpha, \lambda)
, \qquad
m(\alpha, \perp).
\end{equation*}
\end{thm}
\begin{proof}
Set $\tilde{N} := \phi^{-1}(N)$. From Theorem \ref{thm1} (iii) it suffices to compute the principal curvatures of the PF submanifold $\Phi_{\Delta G}^{-1}(\tilde{N})$ of $V_{\mathfrak{g} \oplus \mathfrak{g}}$. For each $x \in \mathfrak{g}$ we denote by $\tilde{x} = (d \phi)^{-1}(x) = \frac{1}{2}(x, - x)$. Set $\tilde{t} = d\phi^{-1}(\mathfrak{t})$ and $\tilde{\Delta} =d \phi^{-1} (\Delta)$. The root space decomposition of $\mathfrak{m} = (\Delta \mathfrak{g})^\perp$ with respect to $\tilde{\mathfrak{t}}$ is given by 
\begin{equation*}
\mathfrak{m} = \mathfrak{m}_0 + \sum_{\tilde{\alpha} \in \tilde{\Delta}} \mathfrak{m}_{\tilde{\alpha}},
\end{equation*}
where
\begin{align*}
&
\mathfrak{m}_{0} = \{\tilde{y} \in \mathfrak{m} \mid \operatorname{ad}(\tilde{\eta})\tilde{y} = 0 \ \ \text{for all} \ \tilde{\eta} \in \tilde{\mathfrak{t}} \},
\\
&
\mathfrak{m}_{\tilde{\alpha}} = \{\tilde{y} \in \mathfrak{m} \mid \operatorname{ad}(\tilde{\eta})^2 \tilde{y}= - \langle \tilde{\alpha}, \tilde{\eta} \rangle^2 \tilde{y} \ \ \text{for all} \ \tilde{\eta} \in \tilde{\mathfrak{t}} \}.
\end{align*}
On the other hand the common eigenspace decomposition by the commuting shape operators $\{A^{\tilde{N}}_{\tilde{\xi}}\}_{\tilde{\xi} \in \tilde{\mathfrak{t}}}$ is 
\begin{equation*}
T_{(e,e)} \tilde{N} = \sum_{\mu \in \Gamma} S_{\mu},
\end{equation*}
where $\Gamma$ is a finite subset of $\tilde{\mathfrak{t}}$ and 
\begin{equation*}
S_{\mu} = \{\tilde{x} \in T_{(e,e)} \tilde{N} \mid A^{\tilde{N}}_{\tilde{\xi}} (\tilde{x}) = \langle \mu , \tilde{\xi} \rangle \tilde{x} \ \text{for all}\ \tilde{\xi} \in \tilde{\mathfrak{t}}\}.
\end{equation*}
 Since the eigenvalues of $A^N_\xi$ and $A^{\tilde{N}}_{\tilde{\xi}}$ coincide it follows that $\Gamma = \{2 \tilde{\lambda} \mid \lambda \in \Lambda\}$.
Thus the assertion follows from the general formula \cite[Theorem 5.2]{M4} together with $\langle \tilde{\alpha} , \tilde{\xi} \rangle = \frac{1}{2} \langle \alpha, \xi\rangle$ and $\langle 2 \tilde{\lambda} ,\tilde{\xi} \rangle = \langle \lambda, \xi \rangle$.
\end{proof}

Considering the case $N = \{e\}$ we obtain the formula for the principal curvatures of fibers of the parallel transport map $\Phi : V_\mathfrak{g} \rightarrow G$ (\cite[Theorem 4.11]{KT93}). Here we choose a maximal abelian subalgebra $\mathfrak{t}$ of $\mathfrak{g}$ and thus $\dim \mathfrak{g}_\alpha =2$.
\begin{cor}[King-Terng \cite{KT93}]\label{corKT}
The principal curvatures of the fiber $\Phi^{-1}(e)$ at $e \in G$ in the direction of $\hat{\xi} \in \hat{\mathfrak{t}}$ are given by 
\begin{equation*}
\{0\}
\cup
\left\{
\left.
\frac{\langle \alpha, \xi \rangle}{2 n \pi}
\ \right|\ 
\alpha \in \Delta^+ \backslash \Delta^+_\xi 
, \ 
n \in \mathbb{Z} \backslash \{0\}
\right\}.
\end{equation*}
The multiplicities are respectively given by 
\begin{equation*}
\infty,  \qquad 2.
\end{equation*}
\end{cor}

\section{Submanifold geometry of orbits of $\sigma$-actions} \label{pcosa}

In this section we study the submanifold geometry of orbits of $\sigma$-actions. The material is based on Ohno \cite{Ohno21} in the case of Hermann actions (see also \cite{M4}).

Let $G$ be a connected compact semisimple Lie group and $\sigma$ an automorphism of $G$. We choose an $\operatorname{Aut}(G)$-invariant inner product of $\mathfrak{g}$ and equip the corresponding bi-invariant Riemannian metric with $G$. 
Take a maximal abelian subalgebra $\mathfrak{t}$ of the fixed point algebra $\mathfrak{g}^\sigma$. The $\sigma$-action is the action of $G(\sigma) = \{(b, \sigma(b))\mid b \in G\}$ on $G$ and is hyperpolar where $\exp \mathfrak{t}$ is a section. 

Consider the root space decomposition 
\begin{equation*}
\mathfrak{g}^\mathbb{C} = \mathfrak{g}(0) + \sum_{\alpha \in \Delta} \mathfrak{g}(\alpha),
\end{equation*}
where 
\begin{align*}
\mathfrak{g}(0) &= \{z \in \mathfrak{g}^\mathbb{C} \mid \operatorname{ad}(z) = 0 \ \ \text{for all $\eta \in \mathfrak{t}$}\},
\\
\mathfrak{g}(\alpha) &= \{z \in \mathfrak{g}^\mathbb{C} \mid \operatorname{ad}(\eta)z = \sqrt{-1} \langle\alpha, \eta \rangle z \ \ \text{for all $\eta \in \mathfrak{t}$}\},
\end{align*}
and $\Delta = \Delta (\sigma)$ is a root system of $\mathfrak{t}$. The real form is 
\begin{equation*}
\mathfrak{g} = \mathfrak{g}_0 + \sum_{\alpha \in \Delta^+} \mathfrak{g}_\alpha,
\end{equation*}
where
\begin{equation*}
\mathfrak{g}_0
=
\mathfrak{g}(0) \cap \mathfrak{g}
, \qquad
\mathfrak{g}_\alpha = (\mathfrak{g}(\alpha) + \mathfrak{g}(- \alpha)) \cap \mathfrak{g}.
\end{equation*}
These are expressed as
\begin{align*}
\mathfrak{g}_0
&=
\{x \in \mathfrak{g} \mid \operatorname{ad}(\eta)x = 0 \ \text{ for all} \ \eta \in \mathfrak{t}  \},
\\
\mathfrak{g}_\alpha
&=
\{x \in \mathfrak{g} \mid  \operatorname{ad}(\eta)^2x = - \langle \alpha, \eta \rangle ^2 x \ \text{ for all} \ \eta \in \mathfrak{t} \}.
\end{align*}
We set $m(\alpha) := \dim \mathfrak{g}_\alpha$.

Consider the eigenspace decomposition of $\sigma$
\begin{equation*}
\mathfrak{g}^\mathbb{C}
=
\sum_{\epsilon \in U(1)} \mathfrak{g}(\epsilon),
\end{equation*}
where
\begin{equation*}
\mathfrak{g}(\epsilon)
=
\{z \in \mathfrak{g}^\mathbb{C} \mid \sigma (z) = \epsilon z\}.
\end{equation*}
Since $\sigma$ commutes with $\operatorname{ad}(\eta)$ we have 
\begin{equation*}
\mathfrak{g}^\mathbb{C}
=
\sum_{\epsilon \in U(1)} \mathfrak{g}(0, \epsilon)
+
\sum_{\alpha \in \Delta}
\sum_{\epsilon \in U(1)} 
\mathfrak{g}(\alpha, \epsilon),
\end{equation*}
where
\begin{equation*}
\mathfrak{g}(0, \epsilon) = \mathfrak{g}(0) \cap \mathfrak{g}(\epsilon)
, \qquad
\mathfrak{g}(\alpha, \epsilon) = \mathfrak{g}(\alpha) \cap \mathfrak{g}(\epsilon).
\end{equation*}
The real form is given by
\begin{equation} \label{decomp1}
\mathfrak{g}
=
\sum_{\epsilon \in U(1)_{\geq 0}} \mathfrak{g}_{0, \epsilon}
+
\sum_{\alpha \in \Delta^+}
\sum_{\epsilon \in U(1)} \mathfrak{g}_{\alpha, \epsilon},
\end{equation}
where $U(1)_{\geq 0} := \{\epsilon \in U(1) \mid \operatorname{Im}(\epsilon) \geq 0\}$ and
\begin{align*}
&
\mathfrak{g}_{0, \epsilon}
=
( \mathfrak{g}(0, \epsilon) + \mathfrak{g}(0, \epsilon^{-1}) \cap \mathfrak{g},
\\
&
\mathfrak{g}_{\alpha, \epsilon}
=
(\mathfrak{g}(\alpha, \epsilon) + \mathfrak{g}(- \alpha, \epsilon^{-1})) \cap \mathfrak{g}.
\end{align*}
We set $m(\alpha, \epsilon) =  \dim \mathfrak{g}_{\alpha, \epsilon}$.
\begin{prop}\label{props}
Let $w \in \mathfrak{g}$. Then the tangent space and the normal space of the orbit $N = G(\sigma) \cdot a$ through $a = \exp w$ are expressed as follows:
\begin{align}
\label{tangent1}
T_{a}N
&=
dl_{a}( \ \ 
\sum_{\substack{\epsilon  \in U(1)_{\geq 0} \\ \epsilon \neq 1}} \mathfrak{g}_{0, \epsilon} 
+
\sum_{\alpha \in \Delta^+}
\sum_{\substack{\epsilon \in U(1) \\ \langle \alpha, w  \rangle + \arg \epsilon \notin 2 \pi\mathbb{Z}}}
\mathfrak{g}_{\alpha, \epsilon}
\ \  ),
\\
T^\perp_{a}N
\label{normal1}
&=
dl_{a}(\ 
\qquad \mathfrak{t}  \qquad 
+ \quad 
\sum_{\alpha \in \Delta^+}
\sum_{\substack{\epsilon \in U(1)\\ \langle \alpha, w  \rangle + \arg \epsilon \in  2\pi\mathbb{Z}}}\mathfrak{g}_{\alpha, \epsilon}
\ \  ).
\end{align}
Moreover \eqref{tangent1} is the common eigenspace decomposition of the family of shape operators $\{A^N_{dl_a (\xi)}\}_{\xi \in \mathfrak{t}}$. In fact
\begin{align*}
dl_a (\mathfrak{g}_{0, \epsilon}): &\ \ \text{the eigenspace associated with the eigenvalue $0$},
\\
dl_a(\mathfrak{g}_{\alpha, \epsilon}): &\ \ \text{the eigenspace associated with the eigenvalue $\textstyle - \frac{\langle \alpha, \xi\rangle}{2} \cot \frac{\langle\alpha, w \rangle + \arg \epsilon}{2}$}.
\end{align*}
\end{prop}
\begin{proof}
Recall that $K := \Delta G$ and $H := G(\sigma)$ are symmetric subgroups of $U := G \times G$ with involution $\theta_K : (b,c) \mapsto (c,b)$ and $\theta_H: (b,c) \mapsto (\sigma^{-1}(c), \sigma(b))$ respectively. Their canonical decompositions are respectively given by 
\begin{equation*}
\mathfrak{u} = \mathfrak{k} + \mathfrak{m}
\quad \text{and} \quad
\mathfrak{u} = \mathfrak{h} + \mathfrak{p},
\end{equation*}
where $\mathfrak{u} = \mathfrak{g} \oplus \mathfrak{g}$, $\mathfrak{k} = \Delta \mathfrak{g}$, $\mathfrak{m} = (\Delta \mathfrak{g})^\perp$, $\mathfrak{h} = \{(x, \sigma(x)) \mid x \in \mathfrak{g}\}$ and $\mathfrak{p} = \{(x, - \sigma(x)) \mid x \in \mathfrak{g}\}$. For each $x \in \mathfrak{g}$ we set $\tilde{x} = (d \phi)^{-1}(x) = \frac{1}{2}(x, -x)$. Note that $\tilde{\mathfrak{t}} = (d \phi)^{-1}(\mathfrak{t})$ is a maximal abelian subspace in $\mathfrak{m} \cap \mathfrak{p}$. Set $\tilde{\Delta} = d\phi^{-1} (\Delta)$. The root space decomposition of $\mathfrak{u}$ with respect to $\tilde{\mathfrak{t}}$ is given by
\begin{equation*}
\mathfrak{u} = \mathfrak{u}(0) + \sum_{\tilde{\alpha} \in \tilde{\Delta}} \mathfrak{u}(\tilde{\alpha}),
\end{equation*}
where
\begin{align*}
\mathfrak{u}(0) &= \{(z_1, z_2) \in \mathfrak{u}^\mathbb{C} \mid \operatorname{ad}(\tilde{\eta}) (z_1, z_2) = 0 \ \ \text{for all $\tilde{\eta} \in \tilde{\mathfrak{t}}$}\},
\\
\mathfrak{u}(\tilde{\alpha}) &= \{(z_1, z_2) \in \mathfrak{u}^\mathbb{C} \mid \operatorname{ad}(\tilde{\eta}) (z_1, z_2) = \langle \tilde{\alpha}, \tilde{\eta} \rangle (z_1, z_2) \ \ \text{for all $\tilde{\eta} \in \tilde{\mathfrak{t}}$}\}.
\end{align*}
Clearly $\mathfrak{u}(0) = \mathfrak{g}(0) \oplus \mathfrak{g}(0)$ and $\mathfrak{u}(\tilde{\alpha}) = \mathfrak{g}(\alpha) \oplus \mathfrak{g}(\alpha)$. Consider the eigenspace decomposition of the composition $\theta_K \circ \theta_H: (b,c) \mapsto (\sigma(b), \sigma^{-1}(c))$: 
\begin{equation*}
\mathfrak{u}^\mathbb{C}
= 
\sum_{\epsilon \in U(1)} \mathfrak{u}(\epsilon),
\end{equation*}
where
\begin{equation*}
\mathfrak{u}(\epsilon) = \{(z_1, z_2) \in \mathfrak{u}^\mathbb{C} \mid \theta_K \circ \theta_H (z_1, z_2) = \epsilon (z_1, z_2)\}.
\end{equation*}
Clearly $\mathfrak{u}(\epsilon) = \mathfrak{g}(\epsilon) \oplus \mathfrak{g}(\epsilon^{-1})$. Similarly to \eqref{decomp1} we have
\begin{equation*}
\mathfrak{u}
=
\sum_{\epsilon \in U(1)_{\geq 0}} \mathfrak{u}_{0, \epsilon}
+
\sum_{\tilde{\alpha} \in \tilde{\Delta}^+}
\sum_{\epsilon \in U(1)} \mathfrak{u}_{\tilde{\alpha}, \epsilon}
\end{equation*}
and therefore we get
\begin{equation*}
\mathfrak{m}
=
\sum_{\epsilon \in U(1)_{\geq 0}} \mathfrak{m}_{0, \epsilon}
+
\sum_{\tilde{\alpha} \in \tilde{\Delta}^+}
\sum_{\epsilon \in U(1)} \mathfrak{m}_{\tilde{\alpha}, \epsilon},
\end{equation*}
where $\mathfrak{m}_{0, \epsilon} = \mathfrak{u}_{0, \epsilon} \cap \mathfrak{m}$ and $\mathfrak{m}_{\tilde{\alpha}, \epsilon} = \mathfrak{u}_{\tilde{\alpha}, \epsilon} \cap \mathfrak{m}$. It is easy to see that $\mathfrak{m}_{0, \epsilon} = \{(x, -x) \mid x \in \mathfrak{g}_{0, \epsilon}\}$ and $\mathfrak{m}_{\tilde{\alpha}, \epsilon} = \{(x, - x) \mid x \in \mathfrak{g}_{\tilde{\alpha}, \epsilon}\}$. Thus $d \phi(\mathfrak{m}_{0, \epsilon}) = \mathfrak{g}_{0, \epsilon}$ and $\phi(\mathfrak{m}_{\tilde{\alpha}, \epsilon}) = \mathfrak{g}_{\alpha, \epsilon}$. Hence the assertion follows from the results of Hermann actions (\cite[p.\ 12]{Ohno21}, \cite[Section 3]{M4}) together with $\langle \tilde{\alpha} , \tilde{\xi} \rangle = \frac{1}{2} \langle \alpha,\xi \rangle$ and $\langle \tilde{\alpha} , \tilde{w} \rangle = \frac{1}{2} \langle \alpha, w \rangle$.
\end{proof}

If $\sigma$ is involutive then we have the $(\pm 1)$-eigenspace decomposition
\begin{equation*}
\mathfrak{g} = \mathfrak{g}^+ + \mathfrak{g}^-,
\end{equation*}
where $\mathfrak{g}^{\pm} = \{x \in \mathfrak{g} \mid \sigma(x) = \pm x\}$. We set
\begin{equation*}
\mathfrak{g}_\alpha^+ := \mathfrak{g}_\alpha \cap \mathfrak{g}^+
\quad \text{and} \quad 
\mathfrak{g}_\alpha^- := \mathfrak{g}_\alpha \cap \mathfrak{g}^-.
\end{equation*}
Since $\mathfrak{g}_{\alpha, 1} = \mathfrak{g}_\alpha^+$ and $\mathfrak{g}_{\alpha, -1} = \mathfrak{g}_\alpha^-$ we obtain:
\begin{cor}\label{cor5.2}
Suppose that $\sigma^2 = \operatorname{id}$. Then the tangent space and the normal space of the orbit $N = G(\sigma) \cdot a$ through $a = \exp w$ are expressed as follows:
\begin{align}
\label{tangent2}
T_{a}N
& =
dl_{a}( \ \ 
\quad \mathfrak{g}_0^- \quad
+
\sum_{\substack{\alpha \in \Delta^+ \\ \langle \alpha, w \rangle \notin 2 \pi \mathbb{Z}}}
\mathfrak{g}_{\alpha}^+
+
\sum_{\substack{\alpha \in \Delta^+ \\ \langle \alpha, w \rangle + \pi \notin 2 \pi \mathbb{Z}}}
\mathfrak{g}_{\alpha}^-
\ \  ),
\\
T^\perp_{a}N
\label{normal2}
&=
dl_{a}(\ 
\qquad \mathfrak{t}  \qquad 
+ \quad 
\sum_{\substack{\alpha \in \Delta^+ \\ \langle \alpha, w \rangle \in 2 \pi \mathbb{Z}}}
\mathfrak{g}_{\alpha}^+
+
\sum_{\substack{\alpha \in \Delta^+ \\ \langle \alpha, w \rangle + \pi \in 2 \pi \mathbb{Z}}}
\mathfrak{g}_{\alpha}^-
\ \  ).
\end{align}
Moreover \eqref{tangent2} is the common eigenspace decomposition of $\{A^N_{dl_a(\xi)}\}_{\xi \in \mathfrak{t}}$. In fact
\begin{align*}
dl_a (\mathfrak{g}_{0}^-): &\ \ \text{the eigenspace associated with the eigenvalue $0$},
\\
dl_a(\mathfrak{g}_{\alpha}^+): &\ \ \text{the eigenspace associated with the eigenvalue $\textstyle - \frac{\langle \alpha, \xi\rangle}{2} \cot \frac{\langle\alpha, w \rangle}{2}$},
\\
dl_a(\mathfrak{g}_{\alpha}^-): &\ \ \text{the eigenspace associated with the eigenvalue $\textstyle \frac{\langle \alpha, \xi\rangle}{2} \tan \frac{\langle\alpha, w \rangle}{2}$}.
\end{align*}
\end{cor}

\begin{cor}
Suppose that $\sigma = \operatorname{id}$. Then the tangent space and the normal space of the orbit $N = \Delta G \cdot a$ through $a = \exp w$ are expressed as follows:
\begin{align}
\label{tangent3}
T_{a}N
& =
dl_{a}( \ \ 
\mathfrak{g}_0 
+
\sum_{\substack{\alpha \in \Delta^+ \\ \langle \alpha, w \rangle \notin 2 \pi \mathbb{Z}}}
\mathfrak{g}_{\alpha}
\ \  ),
\\
T^\perp_{a}N
\label{normal3}
&=
dl_{a}(\ 
\qquad \mathfrak{t}  \qquad 
+ \quad 
\sum_{\substack{\alpha \in \Delta^+ \\ \langle \alpha, w \rangle \in 2 \pi \mathbb{Z}}}
\mathfrak{g}_{\alpha}
\ \  ).
\end{align}
Moreover \eqref{tangent3} is the common eigenspace decomposition of $\{A^N_{dl_a(\xi)}\}_{\xi \in \mathfrak{t}}$. In fact
\begin{align*}
dl_a (\mathfrak{g}_{0}): &\ \ \text{the eigenspace associated with the eigenvalue $0$},
\\
dl_a(\mathfrak{g}_{\alpha}): &\ \ \text{the eigenspace associated with the eigenvalue $\textstyle - \frac{\langle \alpha, \xi\rangle}{2} \cot \frac{\langle\alpha, w \rangle}{2}$}.
\end{align*}
\end{cor}

\section{Principal curvatures of orbits of $P(G, G(\sigma))$-actions} \label{pco}

In this section we derive an explicit formula for the principal curvatures of orbits of $P(G, G (\sigma))$-actions.

As in the last section we let $G$ be a connected compact semisimple Lie group with a bi-invariant metric induced from an $\operatorname{Aut}(G)$-invariant inner product of $\mathfrak{g}$ and $\sigma$ an automorphism of $G$. Take a maximal abelian subalgebra $\mathfrak{t}$ of $\mathfrak{g}^\sigma$. Then $\exp \mathfrak{t}$ is a section of the $\sigma$-action and $\hat{\mathfrak{t}} = \{\hat{x} \mid x \in \mathfrak{t}\}$ is a section of the $P(G, G(\sigma))$-action where $\hat{x}$ denotes the constant path with value $x \in \mathfrak{g}$ (\cite[Theorem 1.2]{Ter95}). Take $w \in \mathfrak{t}$ and set
\begin{align*}
U(1)_\alpha^\top &= \{\epsilon \in U(1) \mid \mathfrak{g}_{\alpha, \epsilon} \neq \{0\}, \ \langle \alpha, w \rangle+ \arg \epsilon  \notin 2 \pi \mathbb{Z}\},
\\
U(1)_\alpha^\perp &= \{\epsilon \in U(1) \mid \mathfrak{g}_{\alpha, \epsilon} \neq \{0\}, \ \langle \alpha, w \rangle + \arg \epsilon  \in 2 \pi \mathbb{Z}\}.
\end{align*}

\begin{thm}\label{thm3}
The principal curvatures of $P(G, G(\sigma)) * \hat{w}$ in the direction of $\hat{\xi} \in \hat{\mathfrak{t}}$ are given by
\begin{align*}
\{0\} 
& \cup
\left\{
\left.
\frac{\langle \alpha, \xi \rangle}{- \langle \alpha, w \rangle - \arg \epsilon + 2 m \pi}
 \ \right| \ 
\alpha \in \Delta^+ \backslash \Delta^+_\xi
, \ 
\epsilon \in U(1)_\alpha^\top
, \ 
m \in \mathbb{Z}
\right\}
\\
& \cup 
\left\{
\left.\frac{\langle \alpha, \xi \rangle}{2 n \pi} 
\ \right| \ 
\alpha \in \Delta^+\backslash \Delta^+_\xi
\textup{ satisfying }
U(1)_\alpha^\perp \neq \emptyset
, \ 
n \in \mathbb{Z} \backslash \{0\}
\right\}.
\end{align*}
The multiplicities are respectively given by
\begin{equation*}
\infty
, \qquad
m(\alpha, \epsilon)
, \qquad
\sum_{\epsilon \in U(1)^\perp_\alpha} m(\alpha, \epsilon).
\end{equation*}
If the orbit is principal then the term $\frac{\langle \alpha, \xi \rangle}{2 n \pi }$ vanishes.
\end{thm}
\begin{proof}
By Corollary \ref{cor3.1} the $P(G, G(\sigma))$-action on $V_\mathfrak{g}$ is conjugate the $P(G \times G, G(\sigma) \times \Delta G)$-action on $V_{\mathfrak{g} \oplus \mathfrak{g}}$. Since the $G(\sigma)$-action on $(G \times G) / \Delta G$ is a Hermann action the assertion follows from \cite[Theorem 6.1]{M4} together with $\langle \tilde{\alpha} , \tilde{\xi} \rangle = \frac{1}{2} \langle \alpha, \xi \rangle$ and $\langle \tilde{\alpha} , \tilde{w} \rangle = \frac{1}{2} \langle \alpha, w \rangle$. (It can be also proven by applying Theorem \ref{thm2} to Proposition \ref{props}.)
\end{proof}

\begin{cor}\label{cor6.2}
Suppose that $\sigma^2 = \operatorname{id}$. Then the principal curvatures of $P(G, G(\sigma)) * \hat{w}$ in the direction of $\hat{\xi} \in \hat{\mathfrak{t}}$ are given by
\begin{align*}
\{0\}
& \cup
\left\{
\left.
\frac{\langle \alpha, \xi \rangle}{- \langle \alpha, w \rangle + 2 m \pi}
 \ \right| \ 
\alpha \in \Delta^+ \backslash \Delta_\xi^+
, \  
\mathfrak{g}_\alpha^+ \neq \{0\}
, \ 
\langle \alpha, w\rangle  \notin 2 \pi \mathbb{Z}, \ m \in \mathbb{Z}
\right\}
\\
& \cup
\left\{
\left.
\frac{\langle \alpha, \xi \rangle}{ - \langle \alpha, w \rangle - \pi + 2 m \pi}
 \ \right| \ 
\alpha \in \Delta^+ \backslash \Delta_\xi^+
, \  
\mathfrak{g}_\alpha^- \neq \{0\}
, \  
\langle \alpha, w\rangle + \pi \notin 2 \pi \mathbb{Z}, \ m \in \mathbb{Z}
\right\}
\\
& \cup
\left\{
\left.
\frac{\langle \alpha, \xi \rangle}{2 n \pi} 
\ \right| \ 
\alpha \in \Delta^+\backslash \Delta_\xi^+
, \ 
\mathfrak{g}_\alpha^+  \neq \{0\}
, \ 
\langle \alpha, w\rangle \in 2 \pi \mathbb{Z}
, \ 
n \in \mathbb{Z} \backslash \{0\}
\right.
\\
&
\left.
\qquad \qquad \text{or} \ \,
\alpha \in \Delta^+ \backslash \Delta_\xi^+
, \ 
\mathfrak{g}_\alpha^- \neq \{0\}
, \ 
\langle \alpha, w\rangle + \pi \in 2 \pi \mathbb{Z}
, \ 
n \in \mathbb{Z} \backslash \{0\}
\right\}.
\end{align*}
The multiplicities are respectively given by
\begin{equation*}
\infty
, \qquad
\dim \mathfrak{g}_\alpha^+
, \qquad
\dim \mathfrak{g}_\alpha^-
\qquad
\dim \mathfrak{g}_\alpha^+ + \dim \mathfrak{g}_\alpha^-.
\end{equation*}
If the orbit is principal then the term $\frac{\langle \alpha, \xi \rangle}{2 n \pi }$ vanishes.
\end{cor}

\begin{cor}\label{corPT}
Suppose that $\sigma = \operatorname{id}$. Then the principal curvatures of $P(G, \Delta G) * \hat{w}$ in the direction of $\hat{\xi} \in \hat{\mathfrak{t}}$ are given by 
\begin{align*}
\{0\} 
& \cup
\left\{
\left.
\frac{\langle \alpha, \xi \rangle}{- \langle \alpha, w \rangle + 2 m \pi}
 \ \right| \ 
\alpha \in \Delta^+\backslash \Delta_\xi^+
, \ 
\langle \alpha, w\rangle \notin 2 \pi \mathbb{Z}, \ m \in \mathbb{Z}
\right\}
\\
& \cup
\left\{
\left.\frac{\langle \alpha, \xi \rangle}{2 n \pi} \ \right| \ 
\alpha \in \Delta^+\backslash \Delta_\xi^+
, \ 
\langle \alpha, w\rangle \in 2 \pi \mathbb{Z},  \ n \in \mathbb{Z} \backslash \{0\}
\right\}.
\end{align*}
The multiplicities are respectively given by 
\begin{equation*}
\infty
, \qquad
\dim \mathfrak{g}_\alpha 
, \qquad
\dim \mathfrak{g}_\alpha.
\end{equation*}
If the orbit is principal then the term $\frac{\langle \alpha, \xi \rangle}{2 n \pi }$ vanishes.
\end{cor}

\begin{rem}
Corollary \ref{corPT} generalizes a result by Palais and Terng in the case of principal $P(G, \Delta G)$-orbits (\cite[Section 5.8]{PT88} and \cite[p.\ 24]{Ter89}).
\end{rem}

\section{The austere property} \label{ausp}
In this section we study conditions for $P(G, G(\sigma))$-orbits to be austere PF submanifolds of $V_\mathfrak{g}$. 

Let $G$ be a connected compact semisimple Lie group with a bi-invariant metric induced from an $\operatorname{Aut}(G)$-invariant inner product on $\mathfrak{g}$ and $\sigma$ be an automorphism of $G$. Consider two conditions for $w \in \mathfrak{g}$:
\begin{enumerate}
\item[(a)] The orbit $N = G(\sigma) \cdot \exp w$ is an austere submanifold of $G$.
\item[(b)] The orbit $\Phi^{-1}(N) = P(G, G(\sigma)) * \hat{w}$ is an austere PF submanifold of $V_\mathfrak{g}$.
\end{enumerate}
Take a maximal abelian subalgebra $\mathfrak{t}$ of $\mathfrak{g}^\sigma$ and denote by $\Delta = \Delta(\sigma)$ the corresponding root system of $\mathfrak{t}$. We show (Theorem II in Introduction): 
\begin{thm}\label{mainthm}\ 
\begin{enumerate}
\item If $\Delta(\sigma)$ is a reduced root system then \textup{(a)} and \textup{(b)} are equivalent. 
\item If $\sigma = \operatorname{id}$ then \textup{(a)} and \textup{(b)} are equivalent.
\item If $\sigma^2 = \operatorname{id}$ then \textup{(a)} implies \textup{(b)}.
\item If $G$ is simple then \textup{(a)} implies \textup{(b)}. 
\end{enumerate}
Here \textup{(b)} does not imply \textup{(a)} in the cases \textup{(iii)} and \textup{(iv)}. In fact, there exists a counterexample.
\end{thm}

As mentioned in the Introduction the above (i)--(iii) follow from Theorem \ref{thm1} and the previous result (\cite{M4}). Thus in this section we prove (iv). (A counterexample to the converse will be shown in Theorem \ref{prop9.2}.) To do this  we need two lemmas. The first one is well-known (\cite[p.\ 44]{LooII}, \cite[Theorem 3.9]{HPTT95}):
\begin{lem}\label{lem1}
Suppose that $G$ is simple. Then there exist $a \in G$ and a diagram  automorphism $\tau$ of $G$ which has order $1$, $2$ or $3$ such that $\sigma = \tau \circ \operatorname{Ad}(a)$.
\end{lem}

\begin{lem}\label{lem2}
Suppose that there exist an automorphism $\tau$ of $G$ and $a \in G$ such that $\sigma = \tau \circ  \operatorname{Ad}(a)$. Then 
\begin{enumerate}
\item the $G(\sigma)$-action is conjugate to the $G(\tau)$-action,

\item the $P(G, G(\sigma))$-action is conjugate to the $P(G, G (\tau))$-action.
\end{enumerate}
\end{lem}
\begin{proof}
(i) Since $G(\sigma) =  (a,e)^{-1} G(\tau) (a,e)$ it follows that the isometry $l_a : G \rightarrow G$ is equivariant via the isomorphism $\operatorname{Ad}(a,e): G(\sigma) \rightarrow G(\tau)$. This proves (i). 

(ii) From the standard arguments in the theory of linear ordinary differential equations there exists a unique $g \in P(G, G \times \{e\})$ satisfying $g(0) = a$. Since $\Psi$ is a group homomorphism it follows that the diagram
\begin{equation*}
\begin{CD}
\mathcal{G} @ > \operatorname{Ad}(g)>> \mathcal{G}
\\
@V\Psi VV @V\Psi VV
\\
G \times G @> \operatorname{Ad}(a,e)>> G \times G
\end{CD}
\end{equation*}
commutes. Since $P(G, L)$ is the inverse image of $L$ under $\Psi$ it follows that $\operatorname{Ad}(g)$ maps $P(G, G(\sigma))$ isomorphically onto $P(G, G(\tau))$. Moreover the isometry $g * : V_\mathfrak{g} \rightarrow V_\mathfrak{g}$ is equivariant via the isomorphism $\operatorname{Ad}(g): P(G, G(\sigma)) \rightarrow P(G, G(\tau))$. This proves (ii).
\end{proof}

We are now in position to prove (iv) of Theorem \textup{\ref{mainthm}}.
\begin{proof}[Proof of Theorem \textup{\ref{mainthm}} \textup{(iv)}] 
From Lemmas \ref{lem1} and \ref{lem2} we can assume without loss of generality that $\sigma$ is a diagram automorphism of $G$ and has order $1$, $2$ or $3$. If $\sigma$ has order $1$ then the  assertion follows from Theorem \textup{\ref{mainthm}} (ii). If $\sigma$ has order $2$ then the assertion follows from Theorem \textup{\ref{mainthm}} (iii). If $\sigma$ has order $3$ then  $\mathfrak{g} = \mathfrak{o}(8)$ and $\sigma$ is the so-called triality automorphism. Take a maximal abelian subalgebra $\mathfrak{t}$ of $\mathfrak{g}^\sigma = \mathfrak{g}_2$. Then the root system $\Delta$ is of type $G_2$ and the assertion follows from Theorem \textup{\ref{mainthm}} (i). 
\end{proof}

\begin{cor}
If $\sigma$ is inner then \textup{(a)} and \textup{(b)} are equivalent.
\end{cor}
\begin{proof}
Since $\sigma$ is inner there exists a maximal abelian subalgebra $\mathfrak{t}$ of $\mathfrak{g}$ which is fixed by $\sigma$, that is, $\mathfrak{t} \subset \mathfrak{g}^\sigma$. Thus the corresponding root system $\Delta$ of $\mathfrak{t}$ is reduced and the assertion follows from (i) of Theorem \ref{mainthm}. 
\end{proof}

\begin{example}\label{example1}
Ikawa \cite{Ika18} classified austere orbits of $\sigma$-actions when $\Delta$ is irreducible and $\sigma$ is involutive. Recently Kimura and Mashimo \cite{KM22} classified Cartan embeddings which are austere submanifolds when $G$ is simple. Applying Theorem \ref{mainthm} to their results we obtain many examples of $P(G, G(\sigma))$-orbits which are austere PF submanifolds of $V_\mathfrak{g}$.  
\end{example}

\section{The weakly reflective property} \label{weakp}
In this section we extend the author's previous results concerning weakly reflective PF submanifolds in Hilbert spaces.

Recall that a submanifold $N$ of a Riemannian manifold $M$ is called \emph{weakly reflective} (\cite{IST09}) if for each normal vector $\xi$ at each $p \in N$ there exists an isometry $\nu_\xi$ of $M$ satisfying 
\begin{equation}\label{wrs}
\nu_\xi(p) = p, \quad \nu_\xi(N) = N, \quad d \nu_\xi(\xi) = - \xi.
\end{equation}
It follows that weakly reflective submanifolds are austere submanifolds. 

The author \cite{M1} extended the concept of weakly reflective submanifolds to the class of PF submanifolds in Hilbert spaces and studied the relation between the following conditions:
\begin{enumerate}
\item[(C)] $N$ is a weakly reflective submanifold of $G/K$.
\item[(D)] $\Phi_K^{-1}(N)$ is a weakly reflective PF submanifold of $V_\mathfrak{g}$.
\end{enumerate}
Here $N$ is a closed submanifold of a compact symmetric space $G/K$. He showed (\cite[Theorem 8]{M1}):
\begin{thm}[\cite{M1}]\label{wrs2}
Let $G$ be  a connected compact semisimple Lie group and $K$ a symmetric subgroup of $G$. Suppose that the bi-invariant Riemannian metric of $G$ is induced by an $\operatorname{Aut}(G)$-invariant inner product of $\mathfrak{g}$ and $G$ acts effectively on $G/K$. Then
\textup{(C)} implies \textup{(D)}.
\end{thm}

It is interesting to remark that here $G$ need not be simple and $N$ need not be an orbit of a Hermann action, unlike in the austere case \cite{M4}. Applying this theorem to examples of weakly reflective submanifolds in $G/K$ he obtained many examples of weakly reflective PF submanifolds in $V_\mathfrak{g}$. We do not know whether (D) implies (C) or not. If the symmetric space $G/K$ is irreducible (or more generally, $G/K$ is a compact isotropy irreducible Riemannian homogeneous space) we can characterize the weakly reflective PF submanifold $\Phi_K^{-1}(N)$ (\cite{M3}). 

In \cite{M1} he also studied the relation between the following conditions:
\begin{enumerate}
\item[(c)] $N$ is a weakly reflective submanifold of $G$.
\item[(d)] $\Phi^{-1}(N)$ is a weakly reflective PF submanifold of $V_\mathfrak{g}$.
\end{enumerate}
Here $N$ is a closed submanifold of a connected compact Lie group $G$. The following theorem claims that (c) implies (d) under strong conditions (\cite[Theorem 7]{M1}):
\begin{thm}[\cite{M1}]\label{wrs3}
Let $G$ be a connected compact semisimple Lie group with a bi-invariant metric. Suppose that $N = L \cdot e$ is the orbit through the identity where $L$ is a closed subgroup of $G \times G$ acting on $G$ by \eqref{action1}. Suppose also that $N$ is a weakly reflective submanifold of $G$ such that for each $\xi \in T^\perp_e N$ the isometry $\nu_\xi$ satisfying \eqref{wrs} can be taken from $\operatorname{Aut}(G)$. Then $\Phi^{-1}(N) = P(G, L) * \hat{0}$ is a weakly reflective PF submanifold of $V_\mathfrak{g}$.
\end{thm}

The following theorem greatly extends the above theorem:
\begin{thm}\label{thm5}
Let $G$ be a connected compact semisimple Lie group with a bi-invariant Riemannian metric induced by an $\operatorname{Aut}(G)$-invariant inner product of $\mathfrak{g}$. Then \textup{(c)} implies \textup{(d)}. 
\end{thm}
\begin{proof}
Set $\tilde{N} := \phi^{-1}(N)$. Since $\tilde{N}$ is weakly reflective it follows from Theorem \ref{wrs2} that $\Phi_{\Delta G}^{-1}(\tilde{N})$ is also weakly reflective. Thus by Theorem \ref{thm1} (iii) the assertion follows.
\end{proof}
\noindent
Here $G$ need not be simple and $N$ need not be an orbit of $\sigma$-action, unlike in the austere case  (Theorem \ref{mainthm}). 

\begin{example}\label{example2}
Kimura and Mashimo \cite{KM22} gave examples of Cartan embeddings which are weakly reflective submanifolds. Applying Theorem \ref{thm5} to their results we obtain $P(G, G(\sigma))$-orbits which are weakly reflective PF submanifold of $V_\mathfrak{g}$. 
\end{example}

\begin{rem}
Taketomi \cite{Tak18} introduced a generalized concept of weakly reflective submanifolds, namely arid submanifolds. The results in this section are also valid in the case of arid submanifolds (see also \cite{M2}).
\end{rem}

\section{Examples and counterexamples}\label{eg}

In this section we show concrete examples of austere PF submanifolds and weakly reflective PF submanifolds which are orbits of a $P(G, G(\sigma))$-action. In particular we show a counterexample to the converse of (iii) and (iv) of Theorem \ref{mainthm}. Throughout this section we will consider the pair 
\begin{equation*}
(G, \sigma) = (SU(2m + 1), \text{the complex conjugation}).
\end{equation*}
This is the only example of a $\sigma$-action whose root system $\Delta(\sigma)$ is non-reduced (\cite[p.\ 558]{Ika18}). 

The canonical decomposition $\mathfrak{g} = \mathfrak{g}^+ + \mathfrak{g}^-$ is given by $\mathfrak{g} = \mathfrak{su}(2m + 1)$, $\mathfrak{g}^+ = \mathfrak{g}^\sigma = \mathfrak{so}(2m + 1)$ and 
\begin{equation*}
\mathfrak{g}^- = \sqrt{-1} \{X \in \operatorname{sym}(2m+1, \mathbb{R}) \mid \operatorname{tr} X = 0\}.
\end{equation*}
We define the $\operatorname{Aut}(G)$-invariant inner product of $\mathfrak{g}$ by
\begin{equation*}
\langle X, Y\rangle = - \frac{1}{2} \operatorname{tr} (X Y)
\quad 
\text{where} \ X, Y \in \mathfrak{g}.
\end{equation*}

A maximal abelian subalgebra of $\mathfrak{g}^+$ is 
\begin{equation*}
\mathfrak{t} =
\left\{
\left.
\left[\begin{array}{cccc}
X_{1} &&&
\vspace{-1mm}
\\
&\ddots&&
\vspace{-1mm}
\\
&& X_{m}&
\vspace{-1mm}
\\
&&& 0
\end{array}
\right]
\in \mathfrak{g}^+
\ \right|\ 
\begin{array}{c}
X_{i} = 
\left[\begin{array}{cc}
0 & - x_i
\\
x_i & 0
\end{array}\right],
\medskip
\\
x_i \in \mathbb{R}
\vspace{-3mm}
\end{array}
\right\}.
\end{equation*}
For each $i = 1, \cdots, m$ we set
\vspace{-1mm}
\begin{equation*}
e_i = 
\begin{blockarray}{cccc}
	&  \text{\footnotesize $i$} & &  \vspace{-2mm} \\
	\begin{block}{[ccc]c}
	& &     &  \medskip 
	\\
	&  \ \ J \ \   & & \!\text{\footnotesize $i$}  \medskip 
	\\
	&   &  &  \\
	\end{block}
\end{blockarray}
\quad \text{where} \ \ 
J = 
\left[\begin{array}{cc}
0 & -1
\\
1 & 0
\end{array}\right].
\vspace{-2mm}
\end{equation*}
Then $\{e_i\}_{i = 1}^m$ is an orthogonal basis of $\mathfrak{t}$.  We set 
\begin{equation*}
\mathfrak{u}
=
\left\{
\left.
\left[\begin{array}{cccc}
Y_{1} &&&
\vspace{-1mm}
\\
&\ddots&&
\vspace{-1mm}
\\
&& Y_{m}&
\vspace{-1mm}
\\
&&& y
\end{array}
\right]
\in \mathfrak{g}^-
\ \right|\ 
\begin{array}{c}
Y_{i} = \sqrt{-1}
\left[\begin{array}{cc}
y_i & 0
\\
0 & y_i
\end{array}\right]
, \ 
y_i \in \mathbb{R},
\medskip
\\
y = - 2 \sqrt{-1}(y_1 + \cdots + y_m)
\vspace{-3mm}
\end{array}
\right\}.
\end{equation*}
We denote by $\mathfrak{g}_{2 e_i}$  the subspace of $\mathfrak{g}^-$ consisting of matrices%
\vspace{-1mm}
\begin{equation*}
\begin{blockarray}{cccc}
	&  \text{\footnotesize $i$} & &  \vspace{-2mm} \\
	\begin{block}{[ccc]c}
	& &     &  \medskip 
	\\
	&  \ \ P \ \   & & \!\text{\footnotesize $i$}  \medskip 
	\\
	&   &  &  \\
	\end{block}
\end{blockarray}
\quad \text{where} \ \ 
P = \sqrt{-1}
\left[\begin{array}{rr}
p & q
\\
q &-p
\end{array}\right]
, \ p, q \in \mathbb{R}
\vspace{-2mm}
\end{equation*} 
and by $\mathfrak{g}_{e_i}$ the subspace of $\mathfrak{g}$ consisting of matrices
\vspace{-1mm}
\begin{equation*}
  \begin{blockarray}{ccccl}
    \begin{block}{ccccl}
     &&\text{\footnotesize$i$}& & 
	\vspace{-1mm}
	\\
    \end{block}
    \begin{block}{[cccc]l}
	&&  &  & 
	\vspace{-1mm}
	\\
	&&  &  & 
	\\
	&&  & \bm{v} & \!\text{\footnotesize $i$}
	\\
	&&  &  & 
	\vspace{-1mm}
	\\
	&&	-{}^t\bar{\bm{v}} & &  
	\vspace{1mm}
	\\
    \end{block}
  \end{blockarray}
\quad
\begin{array}{ll}
\\
\text{where} \ \ 
\bm{v} =  
\left[
\begin{array}{c}
z \\ w
\end{array}
\right]
\in \mathbb{C}^2
\smallskip
\\
\text{lies in the $(2m+1)$-column.}
\end{array}
\vspace{-2mm}
\end{equation*}
For each $1 \leq i < j \leq m$ we denote by $\mathfrak{g}_{e_i \pm e_j}$ the subspace of $\mathfrak{g}$ consisting of matrices
\vspace{-1mm}
\begin{equation*}
\begin{blockarray}{cccccc}
	&  \text{\footnotesize $i$} & & \text{\footnotesize $j$}  & \vspace{-2mm} \\
	\begin{block}{[ccccc]c}
	& &  &  & &  
	\\
	& & & A& & 
	\!\text{\footnotesize $i$}
	\\
	& & & & &  
	\\
	& -{}^t\bar{A} & & & &  
	\!\text{\footnotesize $j$}
	\vspace{-1mm}
	\\
	&&  & & &  
	\\
	\end{block}
\end{blockarray}
\quad \text{where} \ \ 
A = 
\left[\begin{array}{cc}
\alpha & \pm\beta
\\
\beta & \mp\alpha
\end{array}\right]
, \  \alpha, \beta \in \mathbb{C}.
\vspace{-2mm}
\end{equation*}
Then we obtain the root space decomposition
\begin{equation*}
\mathfrak{g} = \mathfrak{t} + \mathfrak{u}  + \sum_{i= 1}^{m} \mathfrak{g}_{2 e_i} + \sum_{i =1}^m \mathfrak{g}_{e_i} 
+ \sum_{1 \leq i < j \leq m}\mathfrak{g}_{e_i + e_j} + \sum_{1 \leq i < j \leq m}\mathfrak{g}_{e_i - e_j}.
\end{equation*}
The root system $\Delta = \{e_i, 2 e_i\}_i \cup \{e_i \pm e_j\}_{i<j}$ is of type $BC$. The dimensions of those spaces are 
\begin{align*}
m, \ m, \ 2, \ 4, \ 4 , \ 4,
\end{align*}
respectively. This decomposition is refined as
\begin{align*}
\mathfrak{g}^+ 
& = 
\mathfrak{t} + \sum_{i = 1}^m \mathfrak{g}_{e_i}^+  + \sum_{1 \leq i < j \leq m}\mathfrak{g}_{e_i + e_j}^+ + \sum_{1 \leq i < j \leq m}\mathfrak{g}_{e_i - e_j}^+,
\\
\mathfrak{g}^-
&=
\mathfrak{g}_0^- + \sum_{i = 1}^m   \mathfrak{g}_{2 e _i} + \sum_{i = 1}^m  \mathfrak{g}_{e_i}^-  + \sum_{1 \leq i < j \leq m}\mathfrak{g}_{e_i + e_j}^- + \sum_{1 \leq i < j \leq m}\mathfrak{g}_{e_i - e_j}^-,
\end{align*}
where $\mathfrak{g}_0^- = \mathfrak{u}$. 

From now on we take $w \in \mathfrak{t}$ and show examples of orbits $G(\sigma) \cdot \exp w$ and $P(G, G(\sigma)) * \hat{w}$ which are austere or weakly reflective. Note that since the actions are hyperpolar it suffices to consider normal vectors $\{dl_{\exp w} (\xi)\}_{\xi \in \mathfrak{t}}$ and $\{\hat{\xi}\}_{\xi \in \hat{\mathfrak{t}}}$ respectively (\cite[Lemma 4.3 and p.\ 458]{IST09}, \cite[Lemma 7.2]{M4}).

\begin{prop}\label{prop9.1}
Let $G$, $\sigma$ be as above. Set 
\begin{equation*}
w := \frac{\pi}{2} \sum_{i = 1}^m e_i.
\end{equation*}
Then 
\begin{enumerate}
\item  the orbit $N = G(\sigma) \cdot \exp w$ is an austere submanifold of $G$,
\item the orbit $\Phi^{-1}(N) = P(G, G(\sigma)) * \hat{w}$ is an austere PF submanifold of $V_\mathfrak{g}$.
\end{enumerate}
\end{prop}

\begin{proof}
(i) Set $a = \exp w$. By Corollary \ref{cor5.2} we have
\begin{align*}
T_{a} N 
&= 
dl_a( \ 
\mathfrak{g}_0^- +\sum_{i = 1}^m \mathfrak{g}_{e_i}^+  + \sum_{1 \leq i < j \leq m}\mathfrak{g}_{e_i + e_j}^+  + \sum_{i = 1}^m  \mathfrak{g}_{e_i}^-  + \sum_{1 \leq i < j \leq m}\mathfrak{g}_{e_i - e_j}^-
\ ) ,
\\
T^\perp_a N 
&=
dl_a ( \ 
\mathfrak{t} 
+ 
\sum_{1 \leq i < j \leq m}\mathfrak{g}_{e_i - e_j}^+
+
\sum_{i = 1}^m   \mathfrak{g}_{2 e _i}
+ 
\sum_{1 \leq i < j \leq m}\mathfrak{g}_{e_i + e_j}^-\ ),
\end{align*}
and the principal curvatures of $N$ in the direction of $dl_a(\xi)$ are expressed as the inner product of $\xi$ and vectors
\begin{equation*}
0, \quad
- \frac{1}{2}e_i
, \quad
0
, \quad
\frac{1}{2} e_i
, \quad
0.
\end{equation*}
Since the set of these vectors are invariant under the multiplication by $(-1)$ it follows that $N$ is an austere submanifold of $G$. (Note that the proof of \cite[Theorem 2.18]{Ika11} on the austere property is not correct when the root system is of type $BC_1$.)

(ii) The assertion follows by applying Theorem \ref{mainthm} (iii) to (i). To describe the principal curvatures explicitly, we give a direct proof here. By  corollary \ref{cor6.2} the principal curvatures of $P(G, G(\sigma)) * \hat{w}$ in the direction of $\hat{\xi}$ is expressed as the inner product of $\xi$ and vectors
\begin{align*}
&
\{0\}
, \ 
\left\{\frac{1}{- \frac{\pi}{2} + 2 m \pi} e_i \right\}_{m \in \mathbb{Z}}
, \ 
\left\{\frac{1}{- \pi + 2 m \pi} (e_i + e_j) \right\}_{m \in \mathbb{Z}}
, \ 
\left\{\frac{1}{- \frac{3}{2}\pi  + 2 m \pi} e_i \right\}_{m \in \mathbb{Z}}
\\
& 
\left\{\frac{1}{- \pi + 2 m \pi}(e_i - e_j) \right\}_{m \in \mathbb{Z}}
, \ 
\left\{
\left.
\frac{1}{2n \pi} \alpha
\ \right|\ 
\alpha = e_i - e_j, \ 2e_i, \ e_i + e_j
\right\}_{n \in \mathbb{Z} \backslash \{0\}}.
\end{align*} 
Clearly the set $\{\frac{1}{2 n \pi } \alpha \}_{n \in \mathbb{Z} \backslash \{0\}}$ is austere (i.e.\ invariant under the multiplication by $(-1)$) for each $\alpha$. By the equality
\begin{equation}
\frac{1}{- \pi + 2 m \pi} (e_i \pm e_j) = (-1) \times \frac{1}{- \pi + 2 (-m+1) \pi} (e_i \pm e_j),
\end{equation}
the set $\{\frac{1}{- \pi + 2 m \pi} (e_i \pm e_j) \}_{m \in \mathbb{Z}}$ is austere. By the equality 
\begin{equation}\label{eq9.1}
\frac{1}{- \frac{\pi}{2} + 2 m \pi} e_i
=
(-1) \times \frac{1}{- \frac{3}{2}\pi + 2  (- m + 1) \pi} e_i,
\end{equation}
the set 
$
\{\frac{1}{- \pi/2 + 2 m \pi} e_i 
, 
\frac{1}{- 3\pi/2  + 2 m \pi} e_i \}_{m \in \mathbb{Z}}
$
is austere. These show that the orbit $P(G, G(\sigma)) * \hat{w}$ is an austere PF submanifold of $V_\mathfrak{g}$.
\end{proof}

The following theorem shows a counterexample to the converse of \textup{(3)} and \textup{(4)} of Theorem \textup{\ref{mainthm}}.
\begin{thm}\label{prop9.2}
Let $G$, $\sigma$ be as above. Set
\begin{equation*}
w := \frac{\pi}{4} \sum_{i =1 }^m e_i.
\end{equation*}
Then 
\begin{enumerate}
\item the orbit $N = G(\sigma) \cdot \exp w$ is a minimal submanifold of $G$, but not an austere submanifold of $G$,
\item the orbit $\Phi^{-1}(N) = P(G, G(\sigma)) * \hat{w}$ is an austere PF submanifold of $V_\mathfrak{g}$.
\end{enumerate}
\end{thm}

\begin{rem}
This counterexample is different from the one given in \cite[Section 9]{M4}. In fact the previous one is not an orbit of a $\sigma$-action. 
The symmetric triads corresponding to those actions are of the same type (II-BC, see \cite{Ika11, Ika18}). However their multiplicities are different.
\end{rem}

\begin{proof}[Proof of Theorem \textup{\ref{prop9.2}}]
(i) Set $a = \exp w$. By Corollary \ref{cor5.2}  we have
\begin{align*}
T_{a} N 
&= 
dl_a( \ 
\mathfrak{g}_0^- +\sum_{i = 1}^m \mathfrak{g}_{e_i}^+  + \sum_{1 \leq i < j \leq m}\mathfrak{g}_{e_i + e_j}^+ 
\\
&  \qquad \quad +
\sum_{i = 1}^m   \mathfrak{g}_{2 e _i} + \sum_{i = 1}^m  \mathfrak{g}_{e_i}^-  + \sum_{1 \leq i < j \leq m}\mathfrak{g}_{e_i + e_j}^- + \sum_{1 \leq i < j \leq m}\mathfrak{g}_{e_i - e_j}^-,
\ ) 
\\
T^\perp_a N 
&=
dl_a ( \ 
\mathfrak{t} + \sum_{1 \leq i < j \leq m}\mathfrak{g}_{e_i - e_j}^+
\ ),
\end{align*}
and the principal curvatures of $N$ in the direction of $dl_a(\xi)$ are expressed as the inner product of $\xi$ and vectors
\begin{equation*}
0, \quad
- \frac{(\sqrt{2} + 1)}{2} e_i
, \quad
-\frac{1}{2}(e_i + e_j)
, \quad
e_i
, \quad 
\frac{(\sqrt{2} -1)}{2} e_i
, \quad
\frac{1}{2}(e_i + e_j)
, \quad
0.
\end{equation*}
Since the sum of these vectors are zero, $N$ is a minimal submanifold. However since the set of those vectors is not invariant under the multiplication by $(-1)$ it is not austere.

(ii) By Corollary \ref{cor6.2} the principal curvatures of $P(G, G(\sigma)) * \hat{w}$ in the direction of $\hat{\xi}$ are the inner product of $\xi$ and vectors
\begin{align*}
&
\{0\}
, \ 
\left\{ \frac{1}{- \frac{\pi}{4} + 2 m \pi} e_i \right\}_{m \in \mathbb{Z}}
, \ 
\left\{ \frac{1}{- \frac{\pi}{2} + 2 m \pi} (e_i + e_j) \right\}_{m \in \mathbb{Z}}
,
\\
&
\left\{ \frac{1}{- \frac{3}{2}\pi + 2 m \pi} 2 e_i\right\}_{m \in \mathbb{Z}}
, \ 
\left\{ \frac{1}{- \frac{5}{4}\pi + 2 m \pi} e_i \right\}_{m \in \mathbb{Z}}
, \ 
\left\{ \frac{1}{- \frac{3}{2}\pi + 2 m \pi} (e_i + e_j) \right\}_{m \in \mathbb{Z}}
,
\\
&
\left\{ \frac{1}{- \pi + 2 m \pi} (e_i - e_j) \right\}_{m \in \mathbb{Z}}
, \ 
\left\{  \frac{1}{2 n \pi } (e_i - e_j)\right\}_{n \in \mathbb{Z} \backslash \{0\}}.
\end{align*}
By the similar arguments as in Proposition \ref{prop9.1} the sets 
$\{\frac{1}{- \pi + 2 m \pi} (e_i - e_j) \}_{m \in \mathbb{Z}}$,  
$\{\frac{1}{2 n \pi } (e_i - e_j)\}_{n \in \mathbb{Z}}$
and 
$\{\frac{1}{- \pi/2 + 2 m \pi} (e_i + e_j), \frac{1}{- 3\pi/2 + 2 m \pi} (e_i + e_j) \}_{m \in \mathbb{Z}}$ 
are austere.  Moreover we have 
\begin{equation*}
\left\{ \frac{1}{ - \frac{3}{2}\pi + 2 m \pi} 2 e_i\right\}_{m \in \mathbb{Z}}
=
\left\{ \frac{1}{- \frac{3}{4}\pi +  2 m \pi} e_i\right\}_{m \in \mathbb{Z}}
\cup
\left\{ \frac{1}{ - \frac{7}{4}\pi +  2 m \pi} e_i\right\}_{m \in \mathbb{Z}}.
\end{equation*}
This together with the equalities
\begin{align*}
\frac{1}{- \frac{\pi}{4} + 2 m \pi} e_i
&=
(-1) \times \frac{1}{- \frac{7}{4}\pi + 2 (- m+1) \pi} e_i,
\\
\frac{1}{- \frac{5}{4}\pi + 2 m \pi} e_i
&=
(-1) \times \frac{1}{- \frac{3}{4}\pi + 2 (- m+1) \pi} e_i,
\end{align*}
shows that  the set 
$
\{
\frac{1}{- \pi/4 + 2 m \pi} e_i 
,
\frac{1}{- 5\pi/4 + 2 m \pi} e_i 
, 
\frac{1}{- 3\pi/2 + 2 m \pi} 2 e_i
\}_{m \in \mathbb{Z}}
$
is austere. These show that the orbit $P(G, G(\sigma)) * \hat{w}$ is an austere PF submanifold of $V_\mathfrak{g}$. 
\end{proof}

The following proposition shows that the austere examples given in Proposition \ref{prop9.1} are actually weakly reflective. Here (i) is base on a result by Ohno \cite[Theorem 5]{Ohno16} (see also  Kimura-Mashimo \cite[Proposition 5.2]{KM22}). 

\begin{prop}\label{prop9.3}
Let $G$, $\sigma$ be as above. Set 
\begin{equation*}
w := \frac{\pi}{2} \sum_{i = 1}^m e_i.
\end{equation*}
Then 
\begin{enumerate}
\item the orbit $N = G(\sigma) \cdot \exp w$ is a weakly reflective submanifold of $G$,
\item the orbit $\Phi^{-1}(N) = P(G, G(\sigma)) * \hat{w}$ is a weakly reflective PF submanifold of $V_\mathfrak{g}$. 
\end{enumerate}
\end{prop}
\begin{proof}
(i) It is easy to see that 
\begin{equation*}
a = \exp w = 
\left[\begin{array}{cccc}
J &&&
\vspace{-1mm}
\\
&\ddots&&
\vspace{-1mm}
\\
&& J&
\vspace{-1mm}
\\
&&& 1
\end{array}
\right]
\quad \text{where} \ \ 
J = 
\left[\begin{array}{cc}
0 & -1
\\
1 & 0
\end{array}\right].
\end{equation*}
Thus
\begin{equation*}
dl_a (\mathfrak{t})
=
\left[\begin{array}{cccc}
R_1 &&&
\vspace{-1mm}
\\
&\ddots&&
\vspace{-1mm}
\\
&& R_m&
\vspace{-1mm}
\\
&&& 0
\end{array}
\right]
\qquad
\text{where}
\ \ 
R_i = 
\left[\begin{array}{cc}
r_i & 0
\\
0 & r_i
\end{array}\right].
\end{equation*}
Define an isometry $\nu: G \rightarrow G$ by 
\begin{equation*}
\nu = (b, \sigma(b))
\quad
\text{where}
 \ \ 
b :=
\left[\begin{array}{cccc}
\sqrt{-1} L &&&
\vspace{-1mm}
\\
&\ddots&&
\vspace{-1mm}
\\
&& \sqrt{-1} L&
\vspace{-1mm}
\\
&&& 1
\end{array}
\right]
, \ \ 
L = 
\left[\begin{array}{cc}
0  &  1
\\
1  &  0
\end{array}\right].
\end{equation*}
Then $\nu \in G(\sigma)$ and  we have
\begin{equation*}
\nu (a) = a
, \quad 
\nu(G(\sigma) \cdot a ) = G(\sigma) \cdot a
\quad \text{and} \quad
d \nu (dl_a \xi) = -dl_a \xi
\end{equation*}
for any $\xi \in \mathfrak{t}$. By homogeneity it follows that $G(\sigma) \cdot \exp w $ is a weakly reflective submanifold of $G$.

(ii) The assertion follows by applying Theorem \ref{thm5} to (i). To express an isometry with respect to normal vectors explicitly we give a direct proof here. Let $q \in \mathcal{G}$ be
\begin{equation*}
q(t) =
\left[\begin{array}{cccc}
Q(t) &&&
\vspace{-1mm}
\\
&\ddots&&
\vspace{-1mm}
\\
&& Q(t)&
\vspace{-1mm}
\\
&&&1
\end{array}
\right]
\qquad
\text{where}
\ \ 
Q(t)= \sqrt{-1}
\left[\begin{array}{rr}
\sin \pi t  &  \cos \pi t
\\
\cos \pi t  &  - \sin \pi t
\end{array}\right]
\end{equation*}
and set
\begin{equation*}
\hat{\nu} := (q*).
\end{equation*}
Clearly $(q(0), q(1)) \in G(\sigma)$ and thus $q \in P(G, G(\sigma))$. Hence 
\begin{equation*}
\nu(P(G, G(\sigma)) * \hat{w}) = P(G, G(\sigma)) * \hat{w}.
\end{equation*}
Moreover it follows from matrix computations that 
\begin{equation*}
q * \hat{w} = \hat{w}
\quad
\text{and}
\quad
d (q*) \hat{\xi}  =  - \hat{\xi},
\end{equation*}
where $d (q*) = \operatorname{Ad}(q)$. Thus by homogeneity it follows that $P(G, G(\sigma)) * \hat{w}$ is a weakly reflective PF submanifold of $V_\mathfrak{g}$.
\end{proof}

\section{Relations to affine Kac-Moody symmetric spaces}\label{AKM}
In this section we study the relation between the canonical isomorphisms defined in Section \ref{paral} and affine Kac-Moody symmetric spaces.

First we review foundations of affine Kac-Moody symmetric spaces following \cite{Hei06}.

Let $G$ be a simply connected compact simple Lie group with Lie algebra $\mathfrak{g}$ and $\sigma$ an automorphism of $G$. The differential of $\sigma$ is still denoted by $\sigma$. Denote by $\langle \cdot, \cdot \rangle$ the inner product of $\mathfrak{g}$ which is the negative of the Killing form of $\mathfrak{g}$. The loop algebra
\begin{equation*}
L(\mathfrak{g}, \sigma)
=
\{u : \mathbb{R} \rightarrow \mathfrak{g} \mid u \in C^\infty , \ u(t + 2 \pi) = \sigma(u(t)) \ \text{for all $t$}\}
\end{equation*}
is a Lie algebra with pointwise bracket.  We equip the inner product $\langle u, v \rangle_{L^2} = \int_0^{2 \pi} \langle u(t),v(t) \rangle dt$ with $L(\mathfrak{g}, \sigma)$.  Denote by $\omega_\lambda$ the cocycle defined by $\omega_\lambda(u, v) = \lambda \langle u', v \rangle_{L^2}$ for $\lambda \in \mathbb{R} \backslash \{0\}$. An \emph{affine Kac-Moody algebra} is a Lie algebra
\begin{equation*}
\hat{L}(\mathfrak{g}, \sigma) := L(\mathfrak{g}, \sigma) + \mathbb{R}c + \mathbb{R}d,
\end{equation*}
where the bracket is defined by
\begin{align*}
&
[u, v] = [u,v] + \omega_\lambda (u,v) c,
\\
&
[d, u] = u',
\\
&
[c, x] = 0,
\end{align*}
where $u, v \in L(\mathfrak{g}, \sigma)$ and $x \in \hat{L}(\mathfrak{g}, \sigma)$. It has the center $\mathbb{R} c$ and the derived algebra $\tilde{L}(\mathfrak{g}, \sigma) := L(\mathfrak{g}, \sigma) + \mathbb{R} c$. If $\sigma_1, \sigma_2 \in \operatorname{Aut}\mathfrak{g}$ are conjugate by an inner automorphism then the corresponding affine Kac-Moody algebras are isomorphic. Thus we can assume that $\sigma$ has finite order. We define the Lorentzian inner product on $\hat{L}(\mathfrak{g}, \sigma)$ by
\begin{equation*}
\langle u + \alpha c + \beta d ,v + \gamma c + \delta d\rangle = \langle u,v \rangle_{L^2} + \alpha \delta + \beta \gamma.
\end{equation*}
Clearly $c, d \perp L(\mathfrak{g}, \sigma)$, $\|c\| = \|d\| = 0$ and $\langle c, d\rangle = 1$. It follows that $\langle [x,y], z\rangle = \langle x, [y,z]\rangle$ for $x,y,z \in \hat{L}(\mathfrak{g}, \sigma)$. 

The twisted loop group
\begin{equation*}
L(G, \sigma)
=
\{g : \mathbb{R} \rightarrow G \mid g \in C^\infty , \ g (t + 2 \pi) = \sigma(g(t)) \ \text{for all $t$}\}
\end{equation*}
with pointwise multiplication is a Fr\'echet Lie group with Lie algebra $L(\mathfrak{g}, \sigma)$. The cocycle $\omega_\lambda$ defines a left-invariant closed $2$-form on $L(G, \sigma)$ and moreover defines a central extension $\tilde{L}(G, \sigma)$ of $L(G, \sigma)$ by the circle $S^1$ for discrete values of $\lambda$ (\cite{PS86}). $\tilde{L}(G, \sigma)$ has Lie algebra $\tilde{L}(\mathfrak{g}, \sigma)$. There exists a unique $\lambda_0$ such that $\tilde{L}(G, \sigma)$ is simply connected. An \emph{affine Kac-Moody group} $\hat{L}(G, \sigma)$ is a Fr\'echet Lie group defined by
\begin{equation*}
\hat{L}(G , \sigma )
:=
S^1 \ltimes  \tilde{L}(G, \sigma).
\end{equation*}
Here the $S^1$-action on $\tilde{L}(G, \sigma)$ is induced by the action on $L(G, \sigma)$ by shifting the parameter of loops. $\hat{L}(G , \sigma )$ is a $2$-torus bundle over $L(G, \sigma)$ and has Lie algebra $\hat{L}(\mathfrak{g}, \sigma)$. We equip the bi-invariant Lorentzian metric on $\hat{L}(G, \sigma)$. Then $\hat{L}(G, \sigma)$ is  a symmetric space where a reflection at the identity is given by $g \mapsto g^{-1}$.

For an involutive automorphism $\hat{\rho}$ of $\hat{G} = \hat{L}(G, \sigma)$ we consider the quotient $\hat{G}/\hat{K}$ by the fixed point subgroup $\hat{K} = \hat{G}^{\hat{\rho}}$. The differential of $\hat{\rho}$ is still denoted by $\hat{\rho}$. The Lie algebra $\hat{\mathfrak{g}} = \hat{L}(\mathfrak{g}, \sigma)$ is decomposed into the $(\pm1)$-eigenspaces $\hat{\mathfrak{g}} = \hat{\mathfrak{k}} + \hat{\mathfrak{m}}$.  Restricting the inner product on $\hat{\mathfrak{g}}$ to $\hat{\mathfrak{m}}$ we equip the $\hat{G}$-invariant metric with $\hat{G}/\hat{K}$. Then $\hat{G}/ \hat{K}$ is a symmetric space where a reflection at $e \hat{K}$ is given by $\hat{g} \hat{K} \mapsto \rho(\hat{g}) \hat{K}$.

From the structure theory of involutions of affine Kac-Moody algebras (\cite{HPTT95, Hei06, HG12}) there are essentially two kinds of involutions, namely 
\begin{enumerate}
\item[(1)] $\hat{\rho}$ satisfies $\hat{\rho}(c) = c$, $\hat{\rho}(d) = d$ and $\hat{\rho}(u)(t) = \rho(u(t))$
where $\rho \in \operatorname{Aut}\mathfrak{g}$, $\rho^2 = \operatorname{id}$ and $\sigma \rho = \rho \sigma$,
\item[(2)] $\hat{\rho}$ satisfies $\hat{\rho}(c) = - c$, $\hat{\rho}(d) = - d$, $\hat{\rho}(u)(t) = \rho(u(-t))$
where $\rho \in \operatorname{Aut}\mathfrak{g}$, $\rho^2 = \operatorname{id}$ and $\sigma \rho = \rho \sigma^{-1}$.
\end{enumerate}
We will always consider the latter one, called the involution of the \emph{second kind}, so that the extension from $L(G, \sigma)$ to $\hat{L}(G, \sigma)$ is not canceled in the quotient.

By definition an \emph{affine Kac-Moody symmetric space} is either an affine Kac-Moody group $\hat{G}$ (the group type) or the symmetric space $\hat{G}/ \hat{K}$ with respect to an involution $\hat{\rho}$ of the second kind. Note that $\hat{G}$ can be written as the quotient $\widehat{G \times G} / (\widehat{G \times G})^{\hat{\rho}}$ where $\widehat{G \times G} = \hat{L}(G \times G, \sigma \times \sigma^{-1})$ is a slight generalization of an affine Kac-Moody group and $\hat{\rho}$ the involution of the second kind defined by
\begin{equation}\label{grpt}
\hat{\rho}(c) = -c, \quad \hat{\rho}(d) = -d , \quad \hat{\rho}(u, v)(t) = (v(-t), u(-t)).
\end{equation}
It was shown that $\hat{G}$ and $\hat{G}/\hat{K}$ are tame Fr\'echet manifolds, where an inverse function theorem is available (\cite{Ham82}). The unique existence theorem of the Levi-Civita connection and the conjugacy theorem of finite dimensional maximal flats are verified for affine Kac-Moody symmetric spaces (\cite{Pop05}). The concept of duality of symmetric spaces is extended to affine Kac-Moody symmetric spaces based on the theory of complex Kac-Moody groups (\cite{Fre11}). The classification of affine Kac-Moody symmetric spaces is essentially equivalent to the classification of involutions of affine Kac-Moody algebras up to conjugation (\cite{HG12}).

Next we review their close relations to hyperpolar PF actions.

Let $\pi : \tilde{L}(G, \sigma) \rightarrow L(G, \sigma)$ denote the projection. For each $\tilde{g} \in \tilde{L}(G, \sigma)$ we write $g = \pi(\tilde{g})$. The \emph{adjoint action} of $\hat{L}(G, \sigma) = S^1 \ltimes \tilde{L}(G, \sigma)$ on $\hat{\mathfrak{g}} = \hat{L}(\mathfrak{g}, \sigma)$ is defined by (\cite{PS86})
\begin{align*}
&
\operatorname{Ad}(\tilde{g}) c  = c,
\\
& \operatorname{Ad}(\tilde{g})d  =  d - g'g^{-1} - \frac{1}{2} \| g'g^{-1}\|^2 c,
\\
&
\operatorname{Ad}(\tilde{g}) u  = gug^{-1} + \langle g' g^{-1} , gug^{-1}\rangle c
\end{align*}
for $\tilde{g} \in \tilde{L}(G, \sigma)$ and 
\begin{equation*}
\operatorname{Ad}(e^{is}) = c
, \quad
\operatorname{Ad}(e^{is}) = d
, \quad
\operatorname{Ad}(e^{is}) u  = u_s 
\end{equation*}
for $e^{is} \in S^1$. Here $u_s(t) := u(s + t)$. For the involution $\hat{\rho}$ of the second kind the canonical decomposition $\hat{\mathfrak{g}} = \hat{\mathfrak{k}} + \hat{\mathfrak{m}}$ is given by
\begin{align*}
& \hat{\mathfrak{k}}  = \{u \in L(\mathfrak{g}, \sigma) \mid \rho(u(-t)) = u(t) \},
\\
&\hat{\mathfrak{m}} = \{u + \alpha c  + \beta d \mid u \in L(\mathfrak{g}, \sigma), \ \rho(u(-t)) = - u(t), \ \alpha , \beta \in \mathbb{R} \}.
\end{align*}
The adjoint action of $\hat{G}$ on $\hat{\mathfrak{g}}$ induces the action of $\hat{K}$ on $\hat{\mathfrak{m}}$, which is called the \emph{isotropy representation} of $\hat{G}/ \hat{K}$. In the group case we define the isotropy representation of $\hat{G}$ to be the induced action of $L(G, \sigma)$ on $\hat{\mathfrak{g}}$.

Since the adjoint action preserves the inner product and the $d$-coefficient it leaves invariant the two-sheeted hyperboloid $\{x \in \hat{L}(\mathfrak{g}, \sigma) \mid \langle x,x \rangle = -1\}$, the hyperplane $\{u + \alpha c  + d\mid u \in L(\mathfrak{g}, \sigma)\}$ and hence their intersection
\begin{equation*}
\operatorname{Hor}(\hat{\mathfrak{g}})
=
\left\{d + u - \frac{\|u\|^2 + 1}{2} c \mid u \in L(\mathfrak{g} , \sigma)\right\},
\end{equation*}
which is geometrically interpreted as a horosphere of codimension $2$. For $x = d +  u - \frac{\|u\|^2 + 1}{2} c$ we have
\begin{equation*}
(e^{is}, \tilde{g}) \cdot x=  \left(d + g * u - \frac{\| g *u\|^2 + 1}{2} c \right)_s,
\end{equation*}
where $g * u = gug^{-1} - g'g^{-1}$ is the gauge transformation. Thus via the isometry
\begin{equation*}
\Gamma : L(\mathfrak{g}, \sigma) \rightarrow \operatorname{Hor}(\hat{\mathfrak{g}}), \qquad u \mapsto d + u - \frac{\|u\|^2 + 1}{2} c
\end{equation*}
$\hat{L}(G, \sigma)$ acts on $L(\mathfrak{g}, \sigma)$ by the gauge transformations.

Recall that two isometric actions of $A_1$ on $X_1$ and of $A_2$ on $X_2$ are called \emph{essentially equivalent} (\cite[p.\ 167]{PS86}) if there exist an injective homomorphism $\phi: A_1 \rightarrow A_2$ and an injective isometry $\varphi:X_1\rightarrow X_2$ which have dense images and satisfy $\varphi(a \cdot p) = \phi(a) \cdot \varphi(p)$ for $a \in A_1$ and $p \in X_1$. For $r > 0$ we set $\mathcal{G}^{r} = H^1([0, r], G)$,  $V_{\mathfrak{g}}^{r}  = H^0([0,r], \mathfrak{g})$ and 
\begin{equation*}
P(G, L)^{r} = \{g \in \mathcal{G}^r \mid (g(0), g(r)) \in L\}
\end{equation*}
 for a closed subgroup $L$ of $G \times G$. Similarly we can define the $P(G,L)^r$-action on $V_\mathfrak{g}^r$ by gauge transformations and the parallel transport map $\Phi^r: V_\mathfrak{g}^r \rightarrow G$. 

The following two propositions show the close relation between affine Kac-Moody symmetric spaces and hyperpolar PF actions (\cite[p.\ 148]{Ter95}, \cite[Proposition 4.14]{HPTT95}). In connection with the formulation of our results we give their proofs here.
\begin{prop}[Terng \cite{Ter95}]\label{prop10.1}
Let $\hat{G} = \hat{L}(G, \sigma)$ be an affine Kac-Moody symmetric space of group type. Then
the isotropy representation restricted to $\operatorname{Hor}(\hat{\mathfrak{g}})$ is essentially equivalent to the $P(G, G(\sigma))^{2 \pi}$-action on $V_\mathfrak{g}^{2 \pi}$. 
\end{prop}
\begin{proof}
The completion of $L(G, \sigma)$ with respect to the $H^1$-metric is 
\begin{align*}
&
\{g: \mathbb{R} \rightarrow G \mid g \in H^1, \ g(t + 2 \pi) = \sigma (g(t)) \ \text{for all $t$}\}
\\
\cong\ &
\{g: [0, 2 \pi] \rightarrow G \mid g \in H^1, \ g(2 \pi) = \sigma (g(0)) \}.
\end{align*}
Moreover the completion of $L(\mathfrak{g}, \sigma)$ with respect to the $H^0$-metric is 
\begin{align*}
&
\{u: \mathbb{R} \rightarrow \mathfrak{g} \mid u \in H^0, \ u(t + 2 \pi) = \sigma (u(t)) \ \text{for all $t$}\}
\\
\cong\ & 
\{u: [0, 2\pi] \rightarrow \mathfrak{g} \mid u \in H^0 \}.
\end{align*}
This proves the proposition.
\end{proof}

\begin{prop}[Heintze-Palais-Terng-Thorbergsson \cite{HPTT95}] \label{prop10.2}
Let $\hat{G}/\hat{K}$ be an affine Kac-Moody symmetric space. Then the isotropy representation restricted to $\operatorname{Hor}(\hat{\mathfrak{g}}) \cap \hat{\mathfrak{m}}$ is essentially equivalent to the $P(G, G^{\rho} \times G^{\sigma \rho})^\pi$-action on $V_\mathfrak{g}^\pi$. Here the inner product of $V_\mathfrak{g}^\pi$ is defined by $\langle u, v\rangle := 2 \int_0^{\pi} \langle u(t), v(t)\rangle dt$.
\end{prop}
\begin{proof}
The completion of $\hat{K}$ with respect to the $H^1$-metric is 
\begin{align*}
&
\{g: \mathbb{R} \rightarrow G \mid g \in H^1, \ g(t + 2 \pi) = \sigma(g(t)), \ \rho(g(-t)) = g(t) \}
\\
\cong\ & 
\{g: [0, 2 \pi] \rightarrow G \mid g \in H^1, \ g(2 \pi) = \sigma(g(0)), \ \rho(\sigma^{-1}g(2 \pi-t)) = g(t) \}
\\
\cong\ &
\{g: [0, \pi] \rightarrow G \mid g \in H^1, \ \rho(\sigma^{-1}g(0)) = \sigma(g(0)), \ \rho(\sigma^{-1}g(\pi)) = g(\pi) \}
\\
= \ &
\{g: [0, \pi] \rightarrow G \mid g \in H^1, \ \sigma^{-1} \rho \sigma^{-1} g(0) = g(0), \ \rho \sigma^{-1}(g(\pi)) = g(\pi) \}
\\
= \ &
\{g: [0, \pi] \rightarrow G \mid g \in H^1, \ \rho g(0) = g(0), \ \sigma \rho(g(\pi)) = g(\pi) \}.
\end{align*}
The completion of  $\Gamma^{-1}(\hat{\mathfrak{m}})$ with respect to the $H^0$-metric is 
\begin{align*}
&
\{u : \mathbb{R} \rightarrow \mathfrak{g} \mid u \in H^0,  \ u(t + 2 \pi) = \sigma (u(t)) , \ \rho(u(-t)) = - u(t)\}
\\
\cong \ &
\{u : [0, 2 \pi] \rightarrow \mathfrak{g} \mid u \in H^0,  \ \rho(\sigma^{-1}u(2\pi-t)) = - u(t)\}
\\
\cong \ &
\{u : [0, \pi] \rightarrow \mathfrak{g} \mid  u \in H^0\}.
\end{align*}
This proves the proposition.
\end{proof}

Finally we focus on the case of group type and show our results.
\begin{prop}\label{cor10.3}
Let $\widehat{G \times G}/ (\widehat{G \times G})^{\hat{\rho}}$ be the affine Kac-Moody symmetric space isomorphic to $\hat{G}$. Then the isotropy representation restricted to the horosphere is essentially equivalent to the $P(G \times G, G(\sigma) \times \Delta G)^{\pi}$-action on $V^{\pi}_{\mathfrak{g} \oplus \mathfrak{g}}$.
\end{prop}
\begin{proof}
Recall that the involution $\hat{\rho}$ was defined by \eqref{grpt}. We consider another involution $\hat{\tau}$ of the second kind defined by
\begin{equation}\label{grpt2}
\hat{\tau}(c) = -c, \quad \hat{\tau}(d) = -d , \quad \hat{\tau}(u, v)(t) = (\sigma^{-1} v(-t), \sigma u(-t)).
\end{equation}
Note that $\hat{\rho}$ and $\hat{\tau}$ are conjugate and thus the corresponding quotients are isomorphic. Then by the similar argument as in Proposition \ref{prop10.2} it follows that the isotropy representation of $\widehat{G \times G}/ (\widehat{G \times G})^{\hat{\tau}}$ restricted to the horosphere is essentially equivalent to the $P(G \times G, G(\sigma) \times \Delta G)^\pi$-action on $V_{\mathfrak{g} \oplus \mathfrak{g}}^\pi$. 
\end{proof}

By the same way as in Section \ref{paral} we define the canonical isomorphisms $\Omega: \mathcal{G}^{2 \pi} \rightarrow \mathcal{G}^{\pi}$ and $\Upsilon: V_{\mathfrak{g}}^{2 \pi} \rightarrow V_{\mathfrak{g} \oplus \mathfrak{g}}^{\pi}$ by
\begin{equation*}
\Omega(g) = (g(t), g(2 \pi - t))
, \quad
\Upsilon(u) = (u(t),  - u(2 \pi -t)).
\end{equation*}

\begin{cor}\label{mainprop2}
Let $\hat{G} = \hat{L}(G, \sigma)$ be an affine Kac-Moody symmetric space of group type and $\widehat{G \times G}/ (\widehat{G \times G})^{\hat{\rho}}$ the quotient isomorphic to $\hat{G}$. Then their isotropy representations restricted to the horospheres are essentially equivalent to the $P(G, G(\sigma))^{2\pi}$-action on $V_{\mathfrak{g}}^{2 \pi}$ and $P(G \times G , G(\sigma) \times \Delta G)^{\pi}$-action on $V_{\mathfrak{g} \oplus \mathfrak{g}}^{\pi}$ respectively and these are conjugate via the canonical isomorphisms $\Omega$ and $\Upsilon$.
\end{cor}
\begin{proof}
The first half of the assertion follows from Propositions \ref{prop10.1} and \ref{cor10.3}. The second half follows from Corollary \ref{cor3.1}.
\end{proof}

This corollary suggests that there is a correspondence between the isomorphism $\hat{G} \cong \widehat{G \times G}/ (\widehat{G \times G})^{\hat{\rho}}$ and the canonical isomorphisms $(\Omega, \Upsilon)$. Let us show this correspondence more explicitly. By conjugacy we can replace the involution $\hat{\rho}$ with $\hat{\tau}$. Consider the map
\begin{equation*}
\lambda: \widehat{G \times G} \rightarrow \hat{G}
\end{equation*}
whose differential is
\begin{equation*}
d\lambda : \widehat{\mathfrak{g} \oplus \mathfrak{g}} \rightarrow \hat{\mathfrak{g}}
, \quad
(u(t),v(t)) + \alpha c + \beta d \mapsto (u(t) - \sigma^{-1} v(-t) ) + \alpha c + \beta d.
\end{equation*}
The inverse image $\lambda^{-1}(\hat{e})$ of the identity $\hat{e}$ is $(\widehat{G \times G})^{\hat{\tau}}$. Thus it induces the isomorphism
\begin{equation*}
\Lambda: \widehat{G \times G} / (\widehat{G \times G})^{\hat{\tau}} \rightarrow \hat{G}.
\end{equation*}
There is an isomorphism between the isotropy subgroups
\begin{equation*}
\varphi : L(G, \sigma) \rightarrow (\widehat{G \times G})^{\hat{\tau}}
, \quad
g(t) \mapsto (g(t), \sigma (g(-t))).
\end{equation*}
The canonical decomposition $\widehat{\mathfrak{g} \oplus \mathfrak{g}} = \hat{\mathfrak{k}} + \hat{\mathfrak{m}}$ with respect to $\hat{\tau}$ is given by 
\begin{align*}
\hat{\mathfrak{k}} & = \{(u(t), \sigma(u(-t))) \mid u \in L(\mathfrak{g}, \sigma)\},
\\
\hat{\mathfrak{m}} &= \{(u(t),  - \sigma u(-t)) + \alpha c + \beta d \mid u \in L(\mathfrak{g}, \sigma), \ \alpha, \beta \in \mathbb{R}  \}.
\end{align*}
There is an isomorphism between the linear subspaces
\begin{equation*}
\psi: \hat{\mathfrak{g}} \rightarrow \hat{\mathfrak{m}}
, \quad
u(t) + \alpha c + \beta d \mapsto (u(t), - \sigma(u(-t))) + \alpha c + \beta d.
\end{equation*}
We define the inner product of $\widehat{\mathfrak{g} \oplus \mathfrak{g}}$ by 
\begin{align*}
\langle (u_1, u_2) + \alpha c + \beta d,  (v_1, v_2) + \gamma c + \delta d\rangle
&=
\frac{1}{2} (\langle u_1,  u_2\rangle_{L^2} + \langle v_1 , v_2\rangle_{L^2}) + \alpha  \delta + \beta \gamma.
\end{align*}
Then the isotropy representations of $L(G, \sigma)$ on $\hat{\mathfrak{g}}$ and of $(\widehat{G \times G})^{\hat{\tau}}$ on $\hat{\mathfrak{m}}$ are conjugate via $\varphi$ and $\psi$. Moreover $\psi$ induces the isometry
\begin{equation*}
\psi: \operatorname{Hor}(\hat{\mathfrak{g}}) \rightarrow \operatorname{Hor}(\widehat{\mathfrak{g} \oplus \mathfrak{g}}) \cap \hat{\mathfrak{m}},
\end{equation*}
which induces
\begin{equation*}
\psi : L(\mathfrak{g}, \sigma) \rightarrow  \Gamma^{-1}(\hat{\mathfrak{m}})
, \quad
u(t)  \mapsto (u(t), - \sigma(u(-t))).
\end{equation*}
Since $\sigma (g(- t)) = g(2 \pi - t)$ and $\sigma (u(- t)) = u (2 \pi -t)$ the diagrams
\begin{equation}\label{comm1}
\begin{CD}
L(G, \sigma) @>\varphi >> \widehat{G \times G}^{\hat{\rho}}
\\
@VVV @VVV
\\
P(G, G(\sigma))^{2 \pi} @>\Omega>> P(G \times G, G(\sigma) \times \Delta G)^{\pi} 
\end{CD}
\end{equation}
and
\begin{equation}\label{comm2}
\begin{CD}
L(\mathfrak{g}, \sigma) @>\psi>> \Gamma^{-1}(\hat{\mathfrak{m}})
\\
@VVV @VVV
\\
V_\mathfrak{g}^{2 \pi} @> \Upsilon>> V_{\mathfrak{g}\oplus \mathfrak{g}}^{\pi} 
\end{CD}
\end{equation}
are commutative, where the vertical arrows denote the injective maps with dense images given in Propositions \ref{prop10.1} and \ref{prop10.2}. 
In this way the isomorphism $\Lambda$ corresponds to $(\Omega, \Upsilon)$.

Moreover we have shown in Theorem  \ref{thm1} that the diagrams 
\begin{equation}\label{comm3}
\begin{CD}
P(G, G(\sigma))^{2 \pi} @>\Omega >> P(G \times G, G(\sigma) \times \Delta G)^\pi
\\
@V (\Psi^G)^{2 \pi} VV @V p^{G \times G} \circ (\Psi^{G \times G})^{\pi}VV 
\\
G(\sigma) @>\operatorname{id} >> G(\sigma)
\end{CD}
\end{equation}
and
\begin{equation}\label{comm4}
\begin{CD}
V_\mathfrak{g}^{2 \pi} @>\Upsilon>> V_{\mathfrak{g} \oplus \mathfrak{g}}^\pi
\\
@V\Phi^{2 \pi} VV @V \Phi_{\Delta G}^{\pi} VV
\\
G @<\phi<< (G \times G) /\Delta G\,,
\end{CD}
\end{equation}
are commutative. This shows that the isomorphisms $(\Omega, \Upsilon)$ correspond to $(\operatorname{id}, \rho)$. 

From these discussions we have: 
\begin{thm}\label{final}
There is a correspondence between:
\begin{enumerate}
\item the isomorphism $\Lambda$ between $\hat{G}$ and $\widehat{G \times G}/ (\widehat{G \times G})^{\hat{\tau}}$,
\item the conjugacy between hyperpolar PF actions of $P(G, G(\sigma))^{2 \pi}$ on $V_\mathfrak{g}^{2 \pi}$ and of $P(G \times G , G(\sigma) \times \Delta G)^{\pi}$ on $V_{\mathfrak{g} \oplus \mathfrak{g}}^{\pi}$ via $\Omega$ and $\Upsilon$ \textup{(}Corollary \textup{\ref{cor3.1}}\textup{)},
\item the conjugacy between the actions of $G(\sigma)$ on $G$ and of $G(\sigma)$ on $(G \times G) / \Delta G$ via $\operatorname{id}$ and $\phi$.
\end{enumerate}
\end{thm}

\section*{Acknowledgements}
The author would like to thank Professor Yoshihiro Ohnita for useful discussions and valuable suggestions. He is also grateful to Professors  Hiroshi Tamaru, Naoyuki Koike, Takashi Sakai and Ian McIntosh for their interests in his work and useful comments. Especially he would like to thank the Osaka Central Advanced Mathematical Institute for constant supports to his work in the excellent research environment at Osaka Metropolitan University.

\end{document}